\documentclass[11pt]{article}
\usepackage[a4paper,margin=1in]{geometry}
\usepackage{amsmath,amssymb,amsthm,mathtools}
\usepackage{mathrsfs}
\usepackage{hyperref}
\hypersetup{hidelinks}
\usepackage{enumitem}
\usepackage{array,tabularx,multirow}

\newcolumntype{C}[1]{>{\centering\arraybackslash}m{#1}}
\newcolumntype{Y}{>{\centering\arraybackslash}X}

\usepackage{xcolor}
\usepackage{tikz}
\usepackage{pgfplots}
\pgfplotsset{compat=1.18}
\usepackage{placeins}
\usetikzlibrary{arrows.meta,calc,patterns,decorations.pathreplacing,positioning}

\newtheorem{definition}{Definition}[section]
\newtheorem{proposition}[definition]{Proposition}
\newtheorem{lemma}[definition]{Lemma}
\newtheorem{theorem}[definition]{Theorem}
\newtheorem{corollary}[definition]{Corollary}
\theoremstyle{remark}

\newcommand{\T}{\mathbb{T}}
\newcommand{\C}{\mathbb{C}}

\newcommand{\rp}{\operatorname{Re}}

\newcommand{\sgn}{\operatorname{sgn}}
\newcommand{\vep}{\varepsilon}
\newcommand{\vphi}{\varphi}

\title{On the Extremal Energy of \\ Complex Unit Gain Dumbbell Graphs}
\author{
Silin Huang \\
\small College of Information Science and Technology, Jinan University, Guangzhou, China \\
\small\texttt{huangsilin@stu.jnu.edu.cn}
\and
Kevin Pereyra\thanks{Corresponding author.} \\
\small Departamento de Matem\'{a}tica, Universidad Nacional de San Luis, San Luis, Argentina \\
\small\texttt{kdpereyra@unsl.edu.ar}
}
\date{}

\begin{document}
\maketitle

\begin{abstract}
We study the extremal energy problem for complex unit gain graphs
whose underlying graph is the dumbbell graph $D_{r,s,\ell}$.
Using switching equivalence,
we reduce the spectrum
to the real parts of the two cycle gains and
obtain an explicit expression of the characteristic polynomial
in terms of matching polynomials of natural subgraphs.
For the bipartite case,
we determine the extremal gain assignments by coefficient comparison.
For the non-bipartite cases,
we analyze the Coulson integral kernels.
Finally, the maximum-energy conditions are determined in all cases,
while the minimum-energy conditions are determined
except when $r$, $s$, and $\ell$ are all odd.
For this remaining case,
we alternatively prove sign restrictions for any improvement over $(0,0)$,
%derive an expression for the pointwise minimum of the Coulson kernel,
%obtain stationarity equations for interior critical points
and prove a Hessian criterion at the origin,
which provides a sufficient condition for $(0,0)$
to fail to be an energy minimizer.
\end{abstract}



\section{Introduction}
Spectral graph theory studies how spectra of graph matrices
reflect structural properties of graphs.
Over the past few decades,
research in this area has gradually evolved
from simple graphs to more graph families
that encode richer information,
such as signed graphs, mixed graphs, and gain graphs.
For instance, Harary~\cite{Harary1953} pioneered
the study of signed graphs to model polarized relationships.
Zaslavsky~\cite{Zaslavsky1989} further introduced
\textit{gain graphs} as a generalization of signed graphs.
Among them, complex unit gain graphs, also called
\textit{$\mathbb{T}$-gain graphs},
are of great importance because
they provide a common framework for simple, signed,
and mixed graphs~\cite{Samanta2025}.
Their spectral theory was formalized by Reff~\cite{Reff2012}.

\begin{definition}
    A \emph{complex unit gain graph} (or \emph{$\T$-gain graph})
    is a pair $\Phi=\Phi(G)=(G,\vphi)$,
    where $G=G(V,E)$ is a finite simple graph and the \emph{gain function}
    \begin{equation}
        \vphi:\vec{E}=\{(u,v),(v,u):u,v\in V,\,u\sim v\}
        \to\T=\{z\in\C:|z|=1\}
    \end{equation}
    assigns a unit complex number to each oriented edge,
    where $u\sim v$ means that $u$ and $v$ are adjacent.
    The gain function $\vphi$ must satisfy
    \begin{equation}
        \vphi(u,v)=\vphi(v,u)^{-1}\quad\text{for all }u\sim v.
    \end{equation}
\end{definition}

A complex unit gain graph has an associated
\emph{Hermitian adjacency matrix}, which is defined to be
a Hermitian matrix whose off-diagonal entries are
the gains of each oriented edge.
Its characteristic polynomial
is called the \emph{characteristic polynomial}
of the corresponding graph.

\begin{definition}
    The \emph{Hermitian adjacency matrix} of $\Phi$,
    denoted $A(\Phi)$, is a Hermitian complex matrix defined as
    \begin{equation}
        A(\Phi)_{uv}=\begin{cases}
            \vphi(u,v), &(u,v)\in\vec{E}, \\
            0, &\text{otherwise}.
        \end{cases}
    \end{equation}
\end{definition}

Spectral properties of Hermitian adjacency matrices of
complex unit gain graphs have been studied in several settings.
For some fundamental spectral results regarding this matrix,
the reader may refer to~\cite{Mehatari2022}.
Among all spectral invariants,
\emph{graph energy}, which was introduced by Gutman~\cite{Gutman1978}
is of great importance as it links the spectrum of a graph
to its structural properties
and has deep connections with mathematical chemistry.
This concept was originally developed for undirected graphs
and has been generalized to the case of complex unit gain graphs;
see~\cite{Belardo2023} for reference.

\begin{definition}
    The \emph{energy} of $\Phi$ is defined as
    \begin{equation}
        E(\Phi)=\sum_{j=1}^n |\lambda_j|,
    \end{equation}
    where $\lambda_1,\dots,\lambda_n$ are the eigenvalues of $A(\Phi)$,
    which are real since $A(\Phi)$ is Hermitian.
\end{definition}

Let $\mathcal{T}_G$ denote the set
of all complex unit gain graphs
defined on a fixed simple graph $G$.
When $G$ is restricted to a certain family of graphs,
natural problems include establishing bounds
for the energy of graphs in $\mathcal{T}_G$
and identifying the elements of $\mathcal{T}_G$
with minimum or maximum energy.
For instance, in the setting of signed graphs,
Bhat and Pirzada~\cite{Bhat2017}
studied the energy of unicyclic signed graphs
and characterized those with minimum energy.
Bhat et al.~\cite{Bhat2018}
later extended this line of research to bicyclic signed graphs,
identifying the graphs with minimum and second-minimum energy.
Wang and Gao~\cite{Wang2020}
further investigated tricyclic signed graphs
of certain types and characterized the extremal sign assignments
attaining minimum energy.

For complex unit gain graphs,
Samanta and Kannan~\cite{Samanta2021}
derived bounds for the energy of
arbitrary complex unit gain graphs and showed that,
among all gain assignments on a complete bipartite graph,
the balanced gain assignment attains the minimum energy.
Samanta and Rajesh~\cite{Samanta2025}
studied unicyclic complex unit gain graphs
and characterized the cycle gain settings
under which such graphs attain the maximum and minimum energy.

Following this line of research,
in this paper we determine the extremal energy of
complex unit gain graphs whose underlying graph
is the dumbbell graph $D_{r,s,\ell}$,
in all cases except for the minimum-energy problem
when $r,s$, and $\ell$ are all odd.
As a classical class of bicyclic graphs,
dumbbell graphs have been widely used in chemical graph theory.
In particular, they are an extremal candidate graph family
in a variety of extremal problems;
readers may refer to~\cite{Furtula2008}.

The result of this paper is as follows.
First, in the bipartite case,
by coefficient comparison,
we give complete extremal results when both cycles are even.
Second, in the non-bipartite cases,
the energy becomes a two-variable integral problem.
By direct analysis,
the maximum-energy problem when both cycles are odd
and both extremal problems in the mixed-parity case
are completely settled.
The only remaining problem is
the minimum-energy problem when $r$, $s$, and $\ell$ are all odd.
We leave this as an open problem; instead, in that case,
we prove sign restrictions for any improvement over $(0,0)$,
derive an expression for the pointwise minimum of the Coulson kernel,
obtain stationarity equations for interior critical points
and prove a Hessian criterion at the origin,
which provides a sufficient condition for $(0,0)$
to fail to be an energy minimizer.

The content of the paper is organized as follows.
Section~\ref{sec:pre} collects preliminaries
and derives the characteristic polynomial expression
for complex unit gain dumbbell graphs.
Sections~\ref{sec:even-even}, \ref{sec:odd-odd}, and \ref{sec:mixed}
treat the even-even, odd-odd, and mixed-parity cases, respectively.
Section~\ref{sec:conclusion} summarizes the results
and indicates some possible directions for future work.

\section{Preliminaries}\label{sec:pre}
Switching equivalent gain graphs are cospectral,
hence equienergetic.

\begin{definition}
    Two gain graphs $\Phi_1=(G,\vphi_1)$ and $\Phi_2=(G,\vphi_2)$
    are \emph{switching equivalent}
    if there exists a unitary diagonal matrix $U$ such that
    \begin{equation}
        A(\Phi_2)=U A(\Phi_1) U^*.
    \end{equation}
    Here, $(\,\cdot\,)^*$ denotes the conjugate transpose of a matrix.
\end{definition}

\begin{theorem}[Reff~{\cite[Lemma~2.2]{Reff2016}}]
    Two gain graphs $\Phi_1=(G,\vphi_1)$ and $\Phi_2=(G,\vphi_2)$
    are switching equivalent if and only if
    the corresponding cycle gains are equal for every cycle
    $C\subseteq G$ in $\Phi_1$ and $\Phi_2$.
\end{theorem}

A gain graph is said to be \emph{balanced}
if the gain of every cycle is $1$.
If a complex unit gain graph $\Phi(G)$ is balanced,
then it is switching equivalent to the complex unit gain graph
with all edge gains equal to $1$.
After this switching, the adjacency matrix is $A(G)$,
the adjacency matrix of its underlying graph $G$.

A dumbbell graph is a bicyclic graph
consisting of two vertex disjoint cycles joined by a path
and is a typical class of bicyclic graphs~\cite{Wang2010}.

\begin{definition}
    For integers $r,s\ge 3$ and $\ell\ge 1$,
    the \emph{dumbbell graph} $D_{r,s,\ell}$
    is obtained from two vertex disjoint cycles $C_r$ and $C_s$
    by choosing one vertex on each cycle
    and joining them by a path of length $\ell$.
    The order of such a graph is $|V(D_{r,s,\ell})|=r+s+\ell-1$.
\end{definition}

The following equivalent form
of Coulson's integral formula
connects the characteristic polynomial
of a graph's adjacency matrix to the graph's energy.

\begin{theorem}[Mateljevi\'{c} et al.~{\cite[Theorem~1]{Mateljevic2010}}]\label{thm:coulson-pre}
    Let $X(t)$ be a real polynomial of degree $n$
    with leading coefficient $1$
    and let $\lambda_1,\lambda_2,\dots,\lambda_n$ be the roots of $X(t)$.
    Then
    \begin{equation}
        \sum_j\left|\rp(\lambda_j)\right|
        =\frac{1}{\pi}\int_{-\infty}^{\infty}\frac{1}{t^2}
        \log\left|t^n X\left(\frac{i}{t}\right)\right|\,dt.
    \end{equation}
\end{theorem}

\begin{theorem}\label{thm:coulson}
    Let $\Phi$ be a $\mathbb{T}$-gain graph
    on $n$ vertices with characteristic polynomial $P_{\Phi}(x)$.
    Then
    \begin{equation}\label{eq:coulson}
        E(\Phi)=\frac{1}{\pi}\int_{-\infty}^{+\infty}\frac{1}{t^2}
        \log\left|t^n P_{\Phi}\left(\frac{i}{t}\right)\right|\,dt
        =\frac{1}{\pi}\int_{0}^{+\infty}\frac{1}{t^2}
        \log\left|t^n P_{\Phi}\left(\frac{i}{t}\right)\right|^2\,dt.
    \end{equation}
\end{theorem}
\begin{proof}
    This follows directly from Theorem~\ref{thm:coulson-pre}
    because $P_{\Phi}(x)$ has real coefficients.
    Moreover, for every real $t\ne 0$,
    \begin{equation}
        P_{\Phi}\left(-\frac{i}{t}\right)
        =\overline{P_{\Phi}\left(\frac{i}{t}\right)}
        \implies
        \left|t^n P_{\Phi}\left(-\frac{i}{t}\right)\right|
        =\left|t^n P_{\Phi}\left(\frac{i}{t}\right)\right|.
    \end{equation}
    Therefore the integrand in~\eqref{eq:coulson}
    is an even function of $t$, and the second equality follows.
\end{proof}

We also use the relation between matchings
and characteristic polynomials.

\begin{definition}
    Let $G=(V,E)$ be a graph.
    A \emph{$j$-matching} in $G$ is a matching
    consisting of exactly $j$ edges,
    in other words, a set $M\subseteq E$ such that $|M|=j$
    and no two edges in $M$ share a common endpoint.
\end{definition}

\begin{definition}
    For a simple graph $G$ on $n$ vertices,
    let $m(G,j)$ be the number of $j$-matchings of $G$.
    Then the \emph{matching polynomial} of $G$ is defined as
    \begin{equation}\label{eq:matching-polynomial-expression}
        m_G(x)=\sum_{j=0}^{\lfloor n/2\rfloor} (-1)^j m(G,j)x^{n-2j}.
    \end{equation}
\end{definition}

\begin{theorem}[Samanta and Rajesh~{\cite[Theorem~4.3]{Samanta2025}}]\label{thm:cugg-char-poly}
    Let $\Phi=(G,\vphi)$ be a complex unit gain graph,
    let $P_{\Phi}(x)$ be the characteristic polynomial of $\Phi$,
    and let $m_G(x)$ be the matching polynomial of $G$. Then
    \begin{equation}
        P_{\Phi}(x)
        =m_G(x)+\sum_K (-2)^{n(K)}
        \prod_{C\in K}\rp(\vphi(C))\,m_{G-K}(x),
    \end{equation}
    where $n(K)$ is the number of cycle components of $K$,
    the summation is over all nontrivial subgraphs
    $K$ of $G$ which are unions of vertex-disjoint cycles,
    and the product is over all cycles $C$ in $K$.
\end{theorem}

By the method of switching equivalence,
we can reduce every problem on $\mathcal{T}_{D_{r,s,\ell}}$
to a problem on two variables.

\begin{proposition}
    Let $\Phi=(D_{r,s,\ell},\vphi)$.
    Then $\Phi$ is switching equivalent to a gain graph
    in which every edge of a fixed spanning tree has gain $1$.
    Consequently, the spectrum and the energy of $\Phi$
    depend only on the two cycle gains
    \begin{equation}
        \gamma_r:=\vphi(C_r),\qquad \gamma_s:=\vphi(C_s).
    \end{equation}
    In particular, they depend only on
    \begin{equation}
        a:=\rp(\gamma_r),\qquad b:=\rp(\gamma_s).
    \end{equation}
\end{proposition}
\begin{proof}
    Suppose that we delete one edge from each cycle.
    Then the remaining graph is a spanning tree.
    Since a complex unit gain tree is switching equivalent
    to any other complex unit gain tree
    of the same underlying graph
    (Samanta and Kannan~\cite[Theorem~3.1]{Samanta2024}),
    the only switching invariants left
    are the gains of the two cycles,
    namely $\gamma_r$ and $\gamma_s$.
    The dependence on $a$ and $b$ follows from
    Theorem~\ref{thm:cugg-char-poly}.
\end{proof}

The switching reduction is illustrated in
Figure~\ref{fig:dumbbell-switching}.
After switching, all edges of a fixed spanning tree
may be assigned gain $1$,
and the two remaining cycle invariants
are the gains $\gamma_r$ and $\gamma_s$.

\begin{figure}[htbp]
    \centering
    \begin{tikzpicture}[
        scale=0.9,
        every node/.style={font=\small},
        vertex/.style={circle,draw,fill=white,inner sep=1.35pt},
        treeedge/.style={line width=0.45pt},
        gainedge/.style={line width=0.95pt},
        >=Stealth
    ]
        % left cycle
        \foreach \i/\ang in {0/0,1/60,2/120,3/180,4/240,5/300}{
            \coordinate (L\i) at ({1.05*cos(\ang)},{1.05*sin(\ang)});
        }
        \foreach \i/\j in {0/1,1/2,2/3,3/4,4/5}{\draw[treeedge] (L\i)--(L\j);}
        \draw[gainedge] (L5)--(L0) node[midway,below right=1pt] {$\gamma_r$};
        \foreach \i in {0,...,5}{\node[vertex] at (L\i) {};}
        \node[draw=none] at (0,-1.45) {$C_r$};

        % right cycle
        \begin{scope}[xshift=6.0cm]
            \foreach \i/\ang in {0/180,1/120,2/60,3/0,4/300,5/240}{
                \coordinate (R\i) at ({1.05*cos(\ang)},{1.05*sin(\ang)});
            }
            \foreach \i/\j in {0/1,1/2,2/3,3/4,4/5}{\draw[treeedge] (R\i)--(R\j);}
            \draw[gainedge] (R5)--(R0) node[midway,below left=1pt] {$\gamma_s$};
            \foreach \i in {0,...,5}{\node[vertex] at (R\i) {};}
            \node[draw=none] at (0,-1.45) {$C_s$};
        \end{scope}

        % connecting path
        \coordinate (P1) at (2.15,0);
        \coordinate (P2) at (2.95,0);
        \coordinate (P3) at (4.05,0);
        \coordinate (P4) at (4.85,0);
        \draw[treeedge] (L0)--(P1)--(P2);
        \draw[treeedge,densely dotted] (P2)--(P3);
        \draw[treeedge] (P3)--(P4)--(R0);
        \foreach \P in {P1,P2,P3,P4}{\node[vertex] at (\P) {};}
        \draw[decorate,decoration={brace,amplitude=4pt}] (1.05,0.38)--(4.95,0.38)
            node[midway,above=5pt,draw=none] {path of length $\ell$};
        \node[draw=none] at (3.0,-0.55) {tree gains equal to $1$};
    \end{tikzpicture}
    \caption{Schematic of a switching normal form
    for a complex unit gain dumbbell graph.}
    \label{fig:dumbbell-switching}
\end{figure}
\FloatBarrier

We now specialize Theorem~\ref{thm:cugg-char-poly} to dumbbell graphs.

\begin{theorem}\label{thm:dumbbell-char-poly}
    For the dumbbell graph $D_{r,s,\ell}$,
    let $\Phi$ be a complex unit gain graph on it.
    Then
    \begin{equation}
        P_\Phi(x)=m_{D_{r,s,\ell}}(x)
        -2a\,m_{D_{r,s,\ell}-C_r}(x)
        -2b\,m_{D_{r,s,\ell}-C_s}(x)
        +4ab\,m_{P_{\ell-1}}(x),
    \end{equation}
    where $a=\rp(\gamma_r)$ and $b=\rp(\gamma_s)$
    are the real parts of the two cycle gains.
\end{theorem}
\begin{proof}
    The nontrivial subgraphs of $G$
    which are unions of vertex-disjoint cycles
    are exactly $C_r$, $C_s$, and $C_r\cup C_s$.
    Applying Theorem~\ref{thm:cugg-char-poly}
    yields the result.
\end{proof}

Theorem~\ref{thm:dumbbell-char-poly}
is the starting point for all later arguments.

\section{The Even-Even Case}\label{sec:even-even}
\subsection{Notation}
Throughout this section,
assume that $r$ and $s$ are even.
For later convenience, write
\begin{equation}
    \vep_r:=(-1)^{r/2},\qquad \vep_s:=(-1)^{s/2},
\end{equation}
and denote
\begin{equation}
    \alpha:=\vep_r\rp(\gamma_r),\qquad \beta:=\vep_s\rp(\gamma_s).
\end{equation}
We remark that $\alpha,\beta\in[-1,1]$.

\subsection{Characteristic Polynomial Coefficients in Simplified Form}
Notice that in this case, $D_{r,s,\ell}$ is bipartite.
The spectrum of every
complex unit gain graph on a bipartite underlying graph
is symmetric about the origin
(Wissing and Van Dam~\cite[Lemma~3.1]{Wissing2025}),
and thus its characteristic polynomial is of the form
\begin{equation}\label{eq:ee-charpoly-expression}
    P_\Phi(x)=x^n+b_2(\Phi)x^{n-2}+b_4(\Phi)x^{n-4}+\cdots.
\end{equation}

For later convenience, denote
\begin{equation}\label{eq:ck-definition}
    c_k(\Phi):=(-1)^k b_{2k}(\Phi).
\end{equation}

\begin{lemma}\label{lem:bipartite-coefficient-comparison}
    Let $\Phi_1$ and $\Phi_2$
    be two complex unit gain graphs on $n$ vertices,
    whose underlying graphs are bipartite.
    Then their characteristic polynomials are
    \begin{equation}\label{eq:bipartite-coeff-comparison-charpoly}
        P_{\Phi_j}(x)
        =\sum_{k=0}^{\lfloor n/2\rfloor}
        (-1)^k c_k(\Phi_j)\,x^{n-2k},
        \qquad j=1,2.
    \end{equation}

    Moreover, if
    \begin{equation}\label{eq:bipartite-coeff-comparison-assumption}
        c_k(\Phi_1)\le c_k(\Phi_2)
        \qquad\text{for all } k=0,1,\dots,\lfloor n/2\rfloor,
    \end{equation}
    then
    \begin{equation}
        E(\Phi_1)\le E(\Phi_2).
    \end{equation}
    Furthermore, if the inequality in~\eqref{eq:bipartite-coeff-comparison-assumption}
    is strict for at least one index $k$, then
    \begin{equation}
        E(\Phi_1)<E(\Phi_2).
    \end{equation}
\end{lemma}
\begin{proof}
    By the preceding discussion,
    the nonzero eigenvalues of $\Phi_j$ occur in pairs
    $\pm \lambda_{j,1},\dots,\pm \lambda_{j,m_j}$,
    and therefore
    \begin{equation}
        P_{\Phi_j}(x)
        =x^{n-2m_j}\prod_{\ell=1}^{m_j}(x^2-\lambda_{j,\ell}^2).
    \end{equation}
    Expanding the product on the right-hand side, we obtain
    \begin{equation}
        c_k(\Phi_j)=
        \begin{cases}
        e_k(\lambda_{j,1}^2,\dots,\lambda_{j,m_j}^2), & 0\le k\le m_j,\\
        0, & m_j<k\le \lfloor n/2\rfloor,
        \end{cases}
    \end{equation}
    where $e_k$ denotes the $k$th elementary symmetric polynomial.

    Now for every $t>0$, by~\eqref{eq:bipartite-coeff-comparison-charpoly},
    \begin{equation}
        t^n P_{\Phi_j}\left(\frac{i}{t}\right)
        =\sum_{k=0}^{\lfloor n/2\rfloor}
        (-1)^k c_k(\Phi_j)\, t^n \left(\frac{i}{t}\right)^{n-2k}
        =i^n \sum_{k=0}^{\lfloor n/2\rfloor} c_k(\Phi_j)\, t^{2k}.
    \end{equation}
    Since all $c_k(\Phi_j)$ are nonnegative, it follows that
    \begin{equation}\label{eq:bipartite-coeff-comparison-abs}
        \left|t^n P_{\Phi_j}\left(\frac{i}{t}\right)\right|
        =\sum_{k=0}^{\lfloor n/2\rfloor} c_k(\Phi_j)\, t^{2k}.
    \end{equation}
    Therefore,~\eqref{eq:bipartite-coeff-comparison-assumption} implies
    \begin{equation}\label{eq:bipartite-coeff-comparison-result}
        \left|t^n P_{\Phi_1}\left(\frac{i}{t}\right)\right|
        \le\left|t^n P_{\Phi_2}\left(\frac{i}{t}\right)\right|
        \qquad\text{for all }t>0.
    \end{equation}
    Applying Theorem~\ref{thm:coulson}
    and using that $\log x$ is increasing on $(0,+\infty)$,
    we obtain
    \begin{equation}
        E(\Phi_1)\le E(\Phi_2).
    \end{equation}

    If $c_{k_0}(\Phi_1)<c_{k_0}(\Phi_2)$ for some $k_0$,
    then by~\eqref{eq:bipartite-coeff-comparison-abs},
    the inequality in~\eqref{eq:bipartite-coeff-comparison-result}
    is strict, because each $t^{2k}$ is positive on $(0,+\infty)$.
    Hence
    \begin{equation}
        E(\Phi_1)<E(\Phi_2).
    \end{equation}
    This completes the proof.
\end{proof}

We now determine the coefficients $c_k(\Phi)$.

\begin{lemma}\label{lem:coeff-even-even}
    For each $k\ge 0$, we have
    \begin{equation}\label{eq:ee-coefficient-expression}
        c_k(\Phi)=M_k-2\alpha R_k-2\beta S_k+4\alpha\beta T_k
        \eqqcolon c_k(\Phi;\alpha,\beta),
    \end{equation}
    where
    \begin{equation}
        \begin{aligned}
            &M_k=m(D_{r,s,\ell},k),\qquad
            &R_k=m(D_{r,s,\ell}-C_r,k-r/2), \\
            &S_k=m(D_{r,s,\ell}-C_s,k-s/2),\qquad
            &T_k=m(P_{\ell-1},k-(r+s)/2)
        \end{aligned}
    \end{equation}
    are real numbers,
    where out-of-range matching indices are interpreted as $0$.
    
    Moreover, it holds that
    \begin{equation}
        R_k\ge 2T_k\qquad\text{and}\qquad S_k\ge 2T_k
    \end{equation}
    for all $k=0,1,\dots,\lfloor n/2\rfloor$.
\end{lemma}

\begin{proof}
    To prove~\eqref{eq:ee-coefficient-expression},
    first, by~\eqref{eq:matching-polynomial-expression}, we have
    \begin{equation}
        m_{D_{r,s,\ell}}(x)
        =\sum_{j=0}^{\lfloor n/2\rfloor} (-1)^j m(D_{r,s,\ell},j) x^{n-2j}.
    \end{equation}
    Also,
    \begin{equation}
        m_{D_{r,s,\ell}-C_r}(x)
        =\sum_{j=0}^{\lfloor (n-r)/2\rfloor} (-1)^j m(D_{r,s,\ell}-C_r,j) x^{(n-r)-2j}.
    \end{equation}
    Since $r$ is even, writing $k=j+r/2$ gives
    \begin{equation}
        m_{D_{r,s,\ell}-C_r}(x)
        =\sum_{k=r/2}^{\lfloor (n-r)/2\rfloor+r/2}
        (-1)^{k-r/2} m(D_{r,s,\ell}-C_r,k-r/2) x^{n-2k},
    \end{equation}
    hence
    \begin{equation}
        -2a\,m_{D_{r,s,\ell}-C_r}(x)
        =\sum_{k=r/2}^{\lfloor n/2\rfloor}
        (-1)^k \left[-2a\,(-1)^{r/2}m(D_{r,s,\ell}-C_r,k-r/2)\right] x^{n-2k}.
    \end{equation}
    Recalling that $\alpha=\varepsilon_r a,\,\varepsilon_r=(-1)^{r/2}$
    and $R_k=m(D_{r,s,\ell}-C_r,k-r/2)$, this becomes
    \begin{equation}
        -2a\,m_{D_{r,s,\ell}-C_r}(x)
        =\sum_{k=r/2}^{\lfloor n/2\rfloor}
        (-1)^k \left(-2\alpha R_k\right) x^{n-2k}.
    \end{equation}

    Similarly, we have
    \begin{equation}
        -2b\,m_{D_{r,s,\ell}-C_s}(x)
        =\sum_{k=s/2}^{\lfloor n/2\rfloor}
        (-1)^k \left(-2\beta S_k\right) x^{n-2k}.
    \end{equation}

    For the last term, since $r+s$ is even and
    $|V(P_{\ell-1})|=\ell-1=n-r-s$,
    by setting $k=j+(r+s)/2$, we obtain
    \begin{equation}
        m_{P_{\ell-1}}(x)
        =\sum_{k=(r+s)/2}^{\lfloor(\ell-1)/2\rfloor+(r+s)/2}
        (-1)^{k-(r+s)/2} m(P_{\ell-1},k-(r+s)/2) x^{n-2k}.
    \end{equation}
    Hence
    \begin{equation}
        4ab\,m_{P_{\ell-1}}(x)
        =\sum_{k=(r+s)/2}^{\lfloor n/2\rfloor} (-1)^k
        \left[4ab\,(-1)^{(r+s)/2}m(P_{\ell-1},k-(r+s)/2)\right]x^{n-2k}.
    \end{equation}
    Because $(-1)^{(r+s)/2}=(-1)^{r/2}(-1)^{s/2}=\vep_r\vep_s$,
    $\alpha=\vep_r a,\,\beta=\vep_s b$ and
    $T_k=m(P_{\ell-1},k-(r+s)/2)$, this is
    \begin{equation}
        4ab\,m_{P_{\ell-1}}(x)
        =\sum_{k=(r+s)/2}^{\lfloor n/2\rfloor}
        (-1)^k \left(4\alpha\beta T_k\right) x^{n-2k}.
    \end{equation}

    Finally, by applying Theorem~\ref{thm:dumbbell-char-poly},
    and by substituting in $m(D_{r,s,\ell},k)=M_k$, we get
    \begin{equation}\label{eq:ee-charpoly-expression2}
        P_\Phi(x)
        =\sum_{k=0}^{\lfloor n/2\rfloor}
        (-1)^k \left(M_k-2\alpha R_k-2\beta S_k+4\alpha\beta T_k\right) x^{n-2k}.
    \end{equation}

    Remark that since we interpret out-of-range matching indices as $0$,
    the three sums may be written with the common range
    $0\le k\le\lfloor n/2\rfloor$.
    By comparing~\eqref{eq:ee-charpoly-expression},
   ~\eqref{eq:ck-definition}, and~\eqref{eq:ee-charpoly-expression2},
    it follows that
    \begin{equation}
        c_k(\Phi)=M_k-2\alpha R_k-2\beta S_k+4\alpha\beta T_k,
    \end{equation}
    which is exactly~\eqref{eq:ee-coefficient-expression}.

    Moreover, since $C_s$ is an even cycle,
    it has exactly two perfect matchings.
    Therefore every matching counted by $T_k$
    extends to a matching counted by $R_k$
    by choosing one of those two perfect matchings on $C_s$.
    Hence $R_k\ge 2T_k$.
    Similarly, we also have $S_k\ge 2T_k$.
    This completes the proof.
\end{proof}

\subsection{Extremal Gain Assignments for Maximum and Minimum Energy}
The lemma in the last section
leads to the monotonicity of the graph energy
with respect to the parameters $\alpha$ and $\beta$.

\begin{theorem}\label{thm:ee-extrema}
    Let $\Phi$ be a complex unit gain graph on $D_{r,s,\ell}$
    where $r$ and $s$ are even.
    Then the minimum of $E(\Phi)$ is attained exactly at
    \begin{equation}
        (\alpha,\beta)=(1,1),
    \end{equation}
    and the maximum is attained exactly at
    \begin{equation}
        (\alpha,\beta)=(-1,-1).
    \end{equation}
\end{theorem}
\begin{proof}
    By Lemma~\ref{lem:coeff-even-even}, for each fixed $k$, we have
    \begin{equation}
        c_k(\Phi)=M_k-2\alpha R_k-2\beta S_k+4\alpha\beta T_k.
    \end{equation}
    Hence
    \begin{equation}
        \frac{\partial c_k}{\partial\alpha}
        =-2R_k+4\beta T_k
        \le -2R_k+4T_k \le 0,\qquad (\alpha,\beta)\in[-1,1]^2,
    \end{equation}
    because $R_k\ge 2T_k$.
    Similarly,
    \begin{equation}
        \frac{\partial c_k}{\partial\beta}
        =-2S_k+4\alpha T_k
        \le -2S_k+4T_k \le 0,\qquad (\alpha,\beta)\in[-1,1]^2,
    \end{equation}
    because $S_k\ge 2T_k$.
    Therefore each $c_k$ is non-increasing in each variable, and thus
    \begin{equation}
        c_k(\Phi;1,1) \le c_k(\Phi) \le c_k(\Phi;-1,-1)
    \end{equation}
    for all $k$.

    We now prove the strictness.
    Suppose $\alpha<1$, then at $k=r/2$ we have
    \begin{equation}
        R_{r/2}=1,\qquad T_{r/2}=0,
    \end{equation}
    hence
    \begin{equation}
        c_{r/2}(\Phi;\alpha,\beta)-c_{r/2}(\Phi;1,\beta)=2(1-\alpha)>0.
    \end{equation}
    Suppose $\alpha=1,\,\beta<1$, then at $k=s/2$ we have
    \begin{equation}
        S_{s/2}=1,\qquad T_{s/2}=0,
    \end{equation}
    hence
    \begin{equation}
        c_{s/2}(\Phi;1,\beta)-c_{s/2}(\Phi;1,1)=2(1-\beta)>0.
    \end{equation}
    Thus, unless $(\alpha,\beta)=(1,1)$,
    at least one coefficient satisfies
    \begin{equation}
        c_k(\Phi;1,1)<c_k(\Phi;\alpha,\beta).
    \end{equation}

    Analogously, suppose $\alpha>-1$, then at $k=r/2$ we have
    \begin{equation}
        c_{r/2}(\Phi;-1,\beta)-c_{r/2}(\Phi;\alpha,\beta)=2(\alpha+1)>0,
    \end{equation}
    while if $\alpha=-1,\,\beta>-1$, then at $k=s/2$ we have
    \begin{equation}
        c_{s/2}(\Phi;-1,-1)-c_{s/2}(\Phi;-1,\beta)=2(\beta+1)>0.
    \end{equation}
    Therefore unless $(\alpha,\beta)=(-1,-1)$,
    at least one coefficient satisfies
    \begin{equation}
        c_k(\Phi;\alpha,\beta)<c_k(\Phi;-1,-1).
    \end{equation}

    By Lemma~\ref{lem:bipartite-coefficient-comparison}, we obtain
    \begin{equation}
        E(\Phi;1,1) \le E(\Phi) \le E(\Phi;-1,-1),
    \end{equation}
    and both equalities are attained exactly at
    the corresponding points.
    This completes the proof.
\end{proof}

It is convenient to write down the result as follows.

\begin{corollary}
    Assume that $r$ and $s$ are even.
    \begin{enumerate}[label=\textup{(\alph*)}]
        \item If $r\equiv s\equiv 0\pmod 4$,
        then the minimum is attained exactly at
        $(\gamma_r,\gamma_s)=(1,1)$
        and the maximum exactly at $(-1,-1)$.

        \item If $r\equiv s\equiv 2\pmod 4$,
        then the minimum is attained exactly at
        $(\gamma_r,\gamma_s)=(-1,-1)$
        and the maximum exactly at $(1,1)$.

        \item If $r\equiv 0\pmod 4$ and $s\equiv 2\pmod 4$,
        then the minimum is attained exactly at
        $(\gamma_r,\gamma_s)=(1,-1)$
        and the maximum exactly at $(-1,1)$.

        \item If $r\equiv 2\pmod 4$ and $s\equiv 0\pmod 4$,
        then the minimum is attained exactly at
        $(\gamma_r,\gamma_s)=(-1,1)$
        and the maximum exactly at $(1,-1)$.
    \end{enumerate}
\end{corollary}

\section{The Odd-Odd Case}\label{sec:odd-odd}
\subsection{Notation}
Throughout this section, assume that $r$ and $s$ are odd.
For later convenience, write
\begin{equation}
    \delta_r:=(-1)^{(r+1)/2},\qquad \delta_s:=(-1)^{(s+1)/2},
\end{equation}
and denote
\begin{equation}
\alpha:=\delta_r\rp(\gamma_r),\qquad \beta:=\delta_s\rp(\gamma_s).
\end{equation}
Then $\alpha,\beta\in[-1,1]$.

\subsection{Matching Polynomials of Paths and Cycles}
Unlike the even-even case, the graph is non-bipartite here,
hence we must rely on analyzing
the Coulson integral~\eqref{eq:coulson} itself.

We start by establishing some fundamental lemmas.
We introduce the recurrence polynomials
\begin{equation}
    f_{-1}(u)=0,\qquad f_0(u)=1,\qquad
    f_j(u)=f_{j-1}(u)+u f_{j-2}(u)\quad(j=1,2,\dots),
\end{equation}
and
\begin{equation}
    g_j(u)=f_j(u)+u f_{j-2}(u)\quad(j=1,2,\dots).
\end{equation}

The following lemma records their connection with
matching polynomials of paths and cycles.

\begin{lemma}[Gutman~\cite{Gutman1979}]\label{lem:matching-poly-of-paths-and-cycles}
    For paths, the recurrence
    \begin{equation}
        m_{P_j}(x)=x\,m_{P_{j-1}}(x)-m_{P_{j-2}}(x),
        \qquad
        m_{P_0}(x)=1,\quad m_{P_1}(x)=x,
    \end{equation}
    implies, by induction on $j$, that
    \begin{equation}
        t^j m_{P_j}(i/t)=i^j f_j(u).
    \end{equation}
    where $u=t^2$.
    For cycles, also by induction on $j$, one has
    \begin{equation}
        m_{C_j}(x)=m_{P_j}(x)-m_{P_{j-2}}(x),
    \end{equation}
    hence
    \begin{equation}
        t^j m_{C_j}(i/t)
        =i^j\left(f_j(u)+u f_{j-2}(u)\right)
        =i^j g_j(u).
    \end{equation}
\end{lemma}

\subsection{Simplified Coulson Formula}
Now, for fixed $r,s,\ell$, we define polynomials
\begin{equation}\label{eq:oo-four-polynomials}
    \begin{aligned}
        &T(u)=f_{\ell-1}(u), \\
        &R(u)=f_{\ell-1}(u)g_s(u)+u f_{\ell-2}(u)f_{s-1}(u), \\
        &S(u)=f_{\ell-1}(u)g_r(u)+u f_{\ell-2}(u)f_{r-1}(u), \\
        &M(u)=\begin{cases}
            g_r(u)g_s(u)+u f_{r-1}(u)f_{s-1}(u), & \ell=1, \\
            \begin{aligned}
                &f_{\ell-1}(u)g_r(u)g_s(u)
                +u f_{\ell-2}(u)\left(f_{r-1}(u)g_s(u)+f_{s-1}(u)g_r(u)\right)\\
                &\qquad +u^2 f_{\ell-3}(u)f_{r-1}(u)f_{s-1}(u),
            \end{aligned} & \ell\ge 2.
        \end{cases}
    \end{aligned}
\end{equation}

The following lemma is obtained by applying
Theorem~\ref{thm:dumbbell-char-poly} and
Lemma~\ref{lem:matching-poly-of-paths-and-cycles}
to the Coulson integral formula.
It expresses the Coulson integral in terms of the polynomials
defined above.

\begin{lemma}\label{lem:odd-odd-kernel}
    Let $\Phi$ be a complex unit gain graph on $D_{r,s,\ell}$
    where $r$ and $s$ are odd.
    For $t>0$, if we put polynomials
    \begin{equation}
        Z(t)=M(t^2),\qquad W(t)=4t^{r+s}T(t^2),
        \qquad X(t)=2t^r R(t^2),\qquad Y(t)=2t^s S(t^2),
    \end{equation}
    where $M,R,S,T$ are defined in~\eqref{eq:oo-four-polynomials},
    and define
    \begin{equation}\label{eq:oo-kernel}
        K(t;\alpha,\beta)=\left(Z(t)-\alpha\beta W(t)\right)^2
        +\left(\alpha X(t)+\beta Y(t)\right)^2,
    \end{equation}
    then we have
    \begin{equation}\label{eq:oo-coulson-expression}
        E(\Phi)=\frac{1}{\pi}\int_0^{+\infty}
        \frac1{t^2}\log K(t;\alpha,\beta)\,dt.
    \end{equation}

    Moreover, we have
    \begin{equation}\label{eq:RSMT}
        R(u)S(u)-M(u)T(u)=(-1)^{\ell}u^{\ell}f_{r-1}(u)f_{s-1}(u),
    \end{equation}
    and as a result,
    \begin{equation}\label{eq:oo-sign-statement}
        \sgn(X(t)Y(t)-Z(t)W(t))=(-1)^{\ell}.
    \end{equation}
\end{lemma}
\begin{proof}
    For~\eqref{eq:oo-coulson-expression} and~\eqref{eq:oo-kernel},
    recalling the Coulson formula in the form
    \begin{equation}
        E(\Phi)=\frac1\pi\int_0^{+\infty} \frac{1}{t^2}
        \log\left|t^nP_\Phi(i/t)\right|^2\,dt,
    \end{equation}
    the goal now is to compute $t^n P_\Phi(i/t)$ explicitly.

    We compute the four matching polynomials
    required for Theorem~\ref{thm:dumbbell-char-poly}.
    Write $u=t^2$.
    A standard identity for the matching polynomial
    of a coalescence graph (Gutman~\cite[Eq.~(3)]{Gutman1986})
    gives
    \begin{equation}
        m_{G\cdot H}(x)
        =m_G(x)m_{H-v}(x)+m_{G-u}(x)m_H(x)-x\,m_{G-u}(x)m_{H-v}(x),
    \end{equation}
    where $G\cdot H$ is obtained by identifying $u\in V(G)$ with $v\in V(H)$.
    Applying this to an endvertex of $P_\ell$
    and a vertex of $C_s$ gives
    \begin{equation}\label{eq:oo-example}
        m_{D_{r,s,\ell}-C_r}(x)
        =m_{P_{\ell-1}}(x)m_{C_s}(x)-m_{P_{\ell-2}}(x)m_{P_{s-1}}(x).
    \end{equation}
    Similarly,
    \begin{equation}
        m_{D_{r,s,\ell}-C_s}(x)
        =m_{P_{\ell-1}}(x)m_{C_r}(x)-m_{P_{\ell-2}}(x)m_{P_{r-1}}(x).
    \end{equation}
    Let $Q_s$ be the graph obtained by
    identifying one end of $P_{\ell+1}$
    with a chosen vertex of $C_s$,
    and let $w$ be the other end of this path.
    Then $D_{r,s,\ell}$ is obtained by
    identifying $w$ with the distinguished vertex of $C_r$.
    Applying the coalescence identity to $C_r$ and $Q_s$,
    and using Lemma~\ref{lem:matching-poly-of-paths-and-cycles},
    yields, for $\ell\ge2$,
    \begin{equation}\label{eq:oo-dumbbell-matching-poly}
        \begin{aligned}
        m_{D_{r,s,\ell}}(x)
        ={}&m_{P_{\ell-1}}(x)m_{C_r}(x)m_{C_s}(x) \\
        &-m_{P_{\ell-2}}(x)\left(m_{P_{r-1}}(x)m_{C_s}(x)+m_{P_{s-1}}(x)m_{C_r}(x)\right) \\
        &+m_{P_{\ell-3}}(x)m_{P_{r-1}}(x)m_{P_{s-1}}(x).
        \end{aligned}
    \end{equation}

    Now we substitute $x=i/t$ in these expressions.
    Take~\eqref{eq:oo-example} for example.
    We have
    \begin{equation}
        \begin{aligned}
            t^{s+\ell-1}m_{D_{r,s,\ell}-C_r}(i/t)
            ={}&t^{s+\ell-1}m_{P_{\ell-1}}(i/t)m_{C_s}(i/t)
            -t^{s+\ell-1}m_{P_{\ell-2}}(i/t)m_{P_{s-1}}(i/t) \\
            ={}&\left(t^{\ell-1}m_{P_{\ell-1}}(i/t)\right)
            \left(t^s m_{C_s}(i/t)\right)
            -t^2\left(t^{\ell-2}m_{P_{\ell-2}}(i/t)\right)
            \left(t^{s-1}m_{P_{s-1}}(i/t)\right) \\
            ={}&i^{\ell-1}f_{\ell-1}(u)\cdot i^s g_s(u)
            -u\,i^{\ell-2}f_{\ell-2}(u)\cdot i^{s-1}f_{s-1}(u) \\
            ={}&i^{s+\ell-1}f_{\ell-1}(u)g_s(u)
            -u\,i^{s+\ell-3}f_{\ell-2}(u)f_{s-1}(u).
        \end{aligned}
    \end{equation}
    Since
    \begin{equation}
        i^{s+\ell-3}=i^{s+\ell-1}i^{-2}=-i^{s+\ell-1},
    \end{equation}
    we obtain
    \begin{equation}
        t^{s+\ell-1}m_{D_{r,s,\ell}-C_r}(i/t)
        =i^{s+\ell-1}\left(f_{\ell-1}(u)g_s(u)+u f_{\ell-2}(u)f_{s-1}(u)\right)
        =i^{s+\ell-1}R(u).
    \end{equation}
    Similarly, we can also obtain
    \begin{equation}
        t^{r+\ell-1}m_{D_{r,s,\ell}-C_s}(i/t)=i^{r+\ell-1}S(u).
    \end{equation}
    Furthermore,
    by~\eqref{eq:oo-dumbbell-matching-poly},~\eqref{eq:oo-four-polynomials}
    and Lemma~\ref{lem:matching-poly-of-paths-and-cycles}, we have
    \begin{equation}
        t^{\ell-1}m_{P_{\ell-1}}(i/t)=i^{\ell-1}T(u),\qquad
        t^n m_{D_{r,s,\ell}}(i/t)=i^n M(u).
    \end{equation}
    Consequently, since $n=r+s+\ell-1$, we finally have
    \begin{equation}
        \begin{aligned}
            &t^n m_{D_{r,s,\ell}-C_r}(i/t)=t^r i^{s+\ell-1}R(u), \qquad
            t^n m_{D_{r,s,\ell}-C_s}(i/t)=t^s i^{r+\ell-1}S(u),\\
            &t^n m_{P_{\ell-1}}(i/t)=t^{r+s} i^{\ell-1}T(u), \qquad
            t^n m_{D_{r,s,\ell}}(i/t)=i^n M(u).
        \end{aligned}
    \end{equation}

    Substitute the above into
    Theorem~\ref{thm:dumbbell-char-poly}. Since
    \begin{equation}
        i^{-r}=(-1)^{(r+1)/2}i=\delta_r i,\quad
        i^{-s}=(-1)^{(s+1)/2}i=\delta_s i,\quad
        i^{-(r+s)}=-\delta_r\delta_s,
    \end{equation}
    we obtain
    \begin{equation}
        \begin{aligned}
            t^n P_\Phi(i/t)
            ={}&i^n M(u)-2a\,t^r i^{s+\ell-1}R(u)
            -2b\,t^s i^{r+\ell-1}S(u)+4ab\,t^{r+s}i^{\ell-1}T(u) \\
            ={}&i^n\left(
            M(u)-i\,(2\delta_r a\,t^r R(u))-i\,(2\delta_s b\,t^s S(u))
            -(4\delta_r\delta_s ab\,t^{r+s}T(u))
            \right).
        \end{aligned}
    \end{equation}
    Using
    \begin{equation}
        \alpha=\delta_r a,\qquad\beta=\delta_s b
    \end{equation}
    together with
    \begin{equation}
        X(t)=2t^r R(u),\qquad
        Y(t)=2t^s S(u),\qquad
        W(t)=4t^{r+s} T(u),\qquad
        Z(t)=M(u),
    \end{equation}
    this simplifies to
    \begin{equation}
        t^nP_\Phi(i/t)
        =i^n\left(Z(t)-\alpha\beta W(t)-i(\alpha X(t)+\beta Y(t))\right).
    \end{equation}
    Therefore
    \begin{equation}
        \left|t^nP_\Phi(i/t)\right|^2
        =\left(Z(t)-\alpha\beta W(t)\right)^2+\left(\alpha X(t)+\beta Y(t)\right)^2
    \end{equation}
    and thus
    \begin{equation}
        E(\Phi)=\frac{1}{\pi}\int_0^{+\infty}\frac{1}{t^2}
        \log\left[\left(Z(t)-\alpha\beta W(t)\right)^2
        +\left(\alpha X(t)+\beta Y(t)\right)^2\right]\,dt,
    \end{equation}
    which is exactly~\eqref{eq:oo-coulson-expression}
    and~\eqref{eq:oo-kernel}.
    
    To prove~\eqref{eq:RSMT}, by~\eqref{eq:oo-four-polynomials},
    a direct expansion shows that
    \begin{equation}
        R(u)S(u)-M(u)T(u)
        =u^2\left(f_{\ell-2}(u)^2-f_{\ell-1}(u)f_{\ell-3}(u)\right)
        f_{r-1}(u)f_{s-1}(u).
    \end{equation}
    Applying the $q$-Euler-Cassini
    formula~\cite[Theorem~2.1]{Cigler2003}
    \begin{equation}
        f_m(u)^2-f_{m+1}(u)f_{m-1}(u)=(-1)^m u^m
    \end{equation}
    with $m=\ell-2$, we obtain
    \begin{equation}
        f_{\ell-2}(u)^2-f_{\ell-1}(u)f_{\ell-3}(u)
        =(-1)^{\ell-2}u^{\ell-2}
        =(-1)^\ell u^{\ell-2}.
    \end{equation}
    Therefore,
    \begin{equation}
        R(u)S(u)-M(u)T(u)
        =(-1)^\ell u^\ell f_{r-1}(u)f_{s-1}(u),
    \end{equation}
    which is exactly~\eqref{eq:RSMT}.
   ~\eqref{eq:oo-sign-statement} then follows because
    \begin{equation}
        X(t) Y(t)-Z(t) W(t)=4t^{r+s}\left(R(u)S(u)-M(u)T(u)\right)
        =(-1)^{\ell}\underbrace{4t^{r+s}\left(u^{\ell}f_{r-1}(u)f_{s-1}(u)\right)}_{\ge 0}.
    \end{equation}
    This completes the proof.
\end{proof}

\subsection{Extremal Gain Assignments for the Maximum Energy}
We now settle the maximum energy problem.

\begin{theorem}\label{thm:odd-odd-maximum}
    Assume that $r$ and $s$ are odd.
    \begin{enumerate}[label=\textup{(\alph*)}]
        \item If $\ell$ is even,
        then the maximum of $E(\Phi)$ is attained exactly at
        \begin{equation}
            (\alpha,\beta)=(1,1)\quad\text{or}\quad
            (\alpha,\beta)=(-1,-1).
        \end{equation}

        \item If $\ell$ is odd,
        then the maximum of $E(\Phi)$ is attained exactly at
        \begin{equation}
            (\alpha,\beta)=(1,-1)\quad\text{or}\quad
            (\alpha,\beta)=(-1,1).
        \end{equation}
    \end{enumerate}
\end{theorem}
\begin{proof}
    By Lemma~\ref{lem:odd-odd-kernel},
    \begin{equation}
        K(t;\alpha,\beta)=\left(Z(t)-\alpha\beta W(t)\right)^2
        +\left(\alpha X(t)+\beta Y(t)\right)^2.
    \end{equation}
    For fixed $\beta$, $K(t;\alpha,\beta)$
    is a strictly convex quadratic in $\alpha$,
    with quadratic coefficient $X(t)^2+\beta^2 W(t)^2>0$.
    For fixed $\alpha$, it is a strictly convex quadratic in $\beta$,
    with quadratic coefficient $Y(t)^2+\alpha^2 W(t)^2>0$.
    Hence any maximum on the square $[-1,1]^2$
    must occur at one of the four corner points $(\pm 1,\pm 1)$.

    At the corner points we have
    \begin{equation}
        K(t;1,1)-K(t;1,-1)=4\left(X(t)Y(t)-Z(t)W(t)\right).
    \end{equation}
    By Lemma~\ref{lem:odd-odd-kernel},
    \begin{equation}
        \sgn\left(X(t)Y(t)-Z(t)W(t)\right)=(-1)^\ell.
    \end{equation}
    Hence, if $\ell$ is even, then
    \begin{equation}
        K(t;-1,-1)=K(t;1,1)>K(t;1,-1)=K(t;-1,1),
    \end{equation}
    whereas if $\ell$ is odd, then
    \begin{equation}
        K(t;-1,-1)=K(t;1,1)<K(t;1,-1)=K(t;-1,1).
    \end{equation}
    This completes the proof.
\end{proof}

\subsection{Extremal Gain Assignments for the Minimum Energy}
We now turn to the minimum energy problem.

\begin{theorem}\label{thm:oo-even-ell-min}
    Assume that $r$ and $s$ are odd \textbf{and that $\ell$ is even}.
    Then the minimum of $E(\Phi)$ is attained exactly at
    \begin{equation}
        (\alpha,\beta)=(0,0).
    \end{equation}
\end{theorem}
\begin{proof}
    By Lemma~\ref{lem:odd-odd-kernel}, since $\ell$ is even,
    we always have
    \begin{equation}
        X(t)Y(t)-Z(t)W(t)>0 \qquad (t>0).
    \end{equation}
    From~\eqref{eq:oo-kernel}, one can compute that
    \begin{equation}\label{eq:oo-ell-odd-fails-here}
        K(t;\alpha,\beta)-Z(t)^2
        =X(t)^2\alpha^2+2(X(t)Y(t)-Z(t)W(t))\alpha\beta+Y(t)^2\beta^2+W(t)^2\alpha^2\beta^2.
    \end{equation}
    The quadratic part in the last equation is positive definite,
    because
    \begin{equation}
        X(t)^2 Y(t)^2-(X(t)Y(t)-Z(t)W(t))^2
        =Z(t)W(t)\left(2X(t)Y(t)-Z(t)W(t)\right),
    \end{equation}
    while
    \begin{equation}
        Z(t) W(t)>0 \qquad\text{and}\qquad
        2X(t)Y(t)-Z(t)W(t)>X(t)Y(t)-Z(t)W(t)>0.
    \end{equation}
    Hence
    \begin{equation}
        K(t;\alpha,\beta)\ge Z(t)^2,
    \end{equation}
    with equality if and only if $(\alpha,\beta)=(0,0)$.
    This completes the proof.
\end{proof}

When $\ell$ is odd,
the quadratic part in~\eqref{eq:oo-ell-odd-fails-here}
is not necessarily positive definite.
As a result, the proof of Theorem~\ref{thm:oo-even-ell-min}
cannot be adapted directly.

Although we do not obtain a full solution to this case,
the following results give several reductions
for the remaining problem.
For the rest of this section,
by $E(\alpha,\beta)$ we denote the energy
of $D_{r,s,\ell}$ where $r,s,\ell$ are all odd
under the parameters $\alpha,\beta$.
We start with a lemma which shows that any energy minimizer
must satisfy $\alpha\beta\ge 0$.

\begin{lemma}\label{lem:oo-odd-sign-restriction}
    Assume that $r$, $s$, and $\ell$ are odd.
    If $\alpha\beta\le 0$ and $(\alpha,\beta)\ne(0,0)$, then
    \begin{equation}
        K(t;\alpha,\beta)>K(t;0,0)
        \qquad\text{for every }t>0.
    \end{equation}
    Consequently,
    \begin{equation}\label{eq:pos-neg-is-not-min-energy}
        E(\alpha,\beta)>E(0,0).
    \end{equation}
    Moreover, it holds that
    \begin{equation}
        E(-\alpha,-\beta)=E(\alpha,\beta).
    \end{equation}
    Hence every point strictly improving on $(0,0)$
    must satisfy $\alpha\beta>0$,
    and every nonzero energy minimizer
    has the companion minimizer $(-\alpha,-\beta)$.
\end{lemma}
\begin{proof}
    Since $\ell$ is odd, by Lemma~\ref{lem:odd-odd-kernel},
    $X(t)Y(t)-Z(t)W(t)<0$.
    Therefore, if $\alpha\beta\le0$, every term in
    \eqref{eq:oo-ell-odd-fails-here} is nonnegative.
    Moreover, if $(\alpha,\beta)\ne(0,0)$,
    at least one of the two terms
    $X(t)^2\alpha^2$ and $Y(t)^2\beta^2$ is strictly positive
    for every $t>0$.
    Hence~\eqref{eq:pos-neg-is-not-min-energy} follows
    from~\eqref{eq:oo-coulson-expression}
    and the monotonicity of $\log x$.
    
    Finally, the symmetry follows immediately from
    \eqref{eq:oo-kernel}, because changing $(\alpha,\beta)$ to
    $(-\alpha,-\beta)$ changes the sign of $\alpha X+\beta Y$ but not its
    square, and leaves $\alpha\beta$ unchanged.
\end{proof}

The next result gives a pointwise minimum expression
for the kernel $K$ when $t$ is fixed.

\begin{theorem}\label{thm:oo-fixed-product-minimization}
    Assume that $r$, $s$, and $\ell$ are odd.
    Fix $t>0$, write $X=X(t)$, $Y=Y(t)$, $Z=Z(t)$, $W=W(t)$.
    Put $x=|\alpha|$, $y=|\beta|$, and $p=xy\in[0,1]$.
    Then, for fixed $p$,
    \begin{equation}\label{eq:fixed-product-minimization}
        \min_{\substack{x,y\in[0,1]\\xy=p}}K(t;\pm x,\pm y)
        =(Z-pW)^2+\mu_t(p),
    \end{equation}
    where the common sign is irrelevant and
    \begin{equation}\label{eq:mu-t-p}
        \mu_t(p)=
        \begin{cases}
            4pXY, & 0\le p\le p_0(t),\\
            (X+pY)^2, & p_0(t)<p\le 1\text{ and }X\le Y,\\
            (pX+Y)^2, & p_0(t)<p\le 1\text{ and }Y<X,
        \end{cases}
    \end{equation}
    with
    \begin{equation}
        p_0(t)=\min\{X/Y,Y/X\}.
    \end{equation}

    Moreover, in the branch $0\le p\le p_0(t)$,
    the product that pointwise minimizes~\eqref{eq:fixed-product-minimization}
    is the projection of
    \begin{equation}\label{eq:pointwise-critical-product}
        p^*(t)=\frac{Z(t)W(t)-2X(t)Y(t)}{W(t)^2}
    \end{equation}
    onto the interval $[0,p_0(t)]$.
\end{theorem}
\begin{proof}
    With $xy=p$, the first term of~\eqref{eq:oo-kernel}
    is $(Z-pW)^2$ and the second term
    is $(xX+yY)^2$, independently of the common sign.
    Thus, as $y=p/x$, we may minimize
    \begin{equation}
        F(x)=xX+\frac{pY}{x},\qquad x\in[p,1].
    \end{equation}
    The unconstrained minimum is attained at
    $x=\sqrt{pY/X}$ and $y=\sqrt{pX/Y}$.
    This point is feasible exactly when $p\le \min\{X/Y,Y/X\}$,
    and then $F_{\min}=2\sqrt{pXY}$, yielding $\mu_t(p)=4pXY$.
    If feasibility fails, the minimum is attained at the
    boundary $x=1$ when $X\le Y$,
    and at the boundary $y=1$ when $Y<X$.
    This gives the two remaining branches.

    In the first branch, by minimizing
    $(Z-pW)^2+4pXY$ with respect to $p$,
    we see that the critical point is exactly
    \eqref{eq:pointwise-critical-product}.
    However, if it lies outside the allowed interval,
    it should be projected to the nearest endpoint.
\end{proof}

One may conjecture that in the symmetric case $r=s$,
the two cycles play symmetric roles
so any nonzero minimizer should lie on the diagonal $\alpha=\beta$.
The next theorem confirms this for $(\alpha,\beta)\in(-1,1)^2$
and records the first-order stationarity conditions,
written as integral equations, for interior minimizers.

\begin{theorem}
    Assume that $r$, $s$, and $\ell$ are odd. Let
    $(\alpha,\beta)\in(-1,1)^2$ be an energy minimizer
    and define
    \begin{equation}
        A(t)=Z(t)-\alpha\beta W(t),\qquad
        B(t)=\alpha X(t)+\beta Y(t),\qquad
        C(t)=A(t)^2+B(t)^2.
    \end{equation}
    Then it holds that
    \begin{equation}\label{eq:EL-alpha-beta}
        \int_0^{+\infty} \frac{-\beta W(t)A(t)+X(t)B(t)}{t^2C(t)}\,dt=0,
        \qquad
        \int_0^{+\infty} \frac{-\alpha W(t)A(t)+Y(t)B(t)}{t^2C(t)}\,dt=0.
    \end{equation}
    Consequently, $(\alpha,\beta)$ satisfies the identities
    \begin{equation}\label{eq:EL-identity-1}
        \alpha^2\int_0^{+\infty}\frac{X(t)^2}{t^2C(t)}\,dt
        =\beta^2\int_0^{+\infty}\frac{Y(t)^2}{t^2C(t)}\,dt
    \end{equation}
    and
    \begin{equation}\label{eq:EL-identity-2}
        \int_0^{+\infty}
        \frac{B(t)^2-2\alpha\beta W(t)A(t)}{t^2C(t)}\,dt=0.
    \end{equation}

    If additionally $r=s$, then every interior energy minimizer
    $(\alpha,\beta)$ with $\alpha\beta>0$ satisfies $\alpha=\beta$.
    Equivalently, in the symmetric case $D_{r,r,\ell}$,
    every nonzero interior minimizer that improves on $(0,0)$
    lies on the diagonal $\alpha=\beta$.
\end{theorem}
\begin{proof}
    Since the point $(\alpha,\beta)$
    lies in the interior of the square,
    an interior minimizer must satisfy
    $\partial_\alpha E(\Phi)=\partial_\beta E(\Phi)=0$.
    We may differentiate the Coulson integral in~\eqref{eq:oo-coulson-expression}
    with respect to $\alpha$ and $\beta$,
    because the kernel $K(t;\alpha,\beta)$
    is a polynomial in the parameters,
    so the differentiation can be justified uniformly
    in a compact neighborhood of the chosen interior point.
    We have
    \begin{align}
        \frac{\partial K(t)}{\partial\alpha}
        &=2A(t)(-\beta W(t))+2B(t)X(t)=2(-\beta W(t)A(t)+X(t)B(t)), \\
        \frac{\partial K(t)}{\partial\beta}
        &=2A(t)(-\alpha W(t))+2B(t)Y(t)=2(-\alpha W(t)A(t)+Y(t)B(t)).
    \end{align}
    This gives~\eqref{eq:EL-alpha-beta}.
    Multiplying the first equation
    in~\eqref{eq:EL-alpha-beta} by $\alpha$,
    the second by $\beta$, and subtracting them gives
    \begin{equation}
        \int_0^{+\infty}
        \frac{(\alpha X(t)-\beta Y(t))(\alpha X(t)+\beta Y(t))}{t^2C(t)}\,dt=0,
    \end{equation}
    which is exactly~\eqref{eq:EL-identity-1}.
    Adding the same two equations gives~\eqref{eq:EL-identity-2}.
    If $r=s$, then $X=Y$, so~\eqref{eq:EL-identity-1} implies
    $\alpha^2=\beta^2$.
    By Lemma~\ref{lem:oo-odd-sign-restriction},
    $\alpha\beta>0$, therefore $\alpha=\beta$.
\end{proof}

We now give a Hessian criterion
which provides a sufficient condition for $(0,0)$
to fail to be an energy minimizer.

\begin{theorem}\label{thm:oo-hessian}
    Assume that $r$, $s$, and $\ell$ are odd.
    Write three integrals
    \begin{equation}
        I_1=\int_0^{+\infty}\frac{X(t)^2}{t^2Z(t)^2}\,dt,
        \qquad
        I_2=\int_0^{+\infty}\frac{Y(t)^2}{t^2Z(t)^2}\,dt,
        \qquad
        I_3=\int_0^{+\infty}\frac{Z(t)W(t)-X(t)Y(t)}{t^2Z(t)^2}\,dt.
    \end{equation}
    Then the Hessian matrix of $E$ at $(0,0)$ is
    \begin{equation}\label{eq:hessian-matrix}
        H(0,0)=\frac{2}{\pi}
        \begin{bmatrix}
            I_1 & -I_3\\
            -I_3 & I_2
        \end{bmatrix}.
    \end{equation}
    Consequently, $(0,0)$ is a strict local minimizer
    if $I_1I_2>I_3^2$,
    is degenerate to second order if $I_1I_2=I_3^2$,
    and is a saddle point if $I_1I_2<I_3^2$.

    Specifically,
    if $I_3^2>I_1 I_2$, then $(\alpha,\beta)=(0,0)$
    is not even a local minimizer,
    and hence the energy at that point cannot be
    a global minimum.
\end{theorem}
\begin{proof}
    Recall that, for $\theta=(\alpha,\beta)$ near the origin,
    \begin{equation}
        K(t;\alpha,\beta)=Z(t)^2+X(t)^2\alpha^2
        +2\left(X(t)Y(t)-Z(t)W(t)\right)\alpha\beta
        +Y(t)^2\beta^2+W(t)^2\alpha^2\beta^2.
    \end{equation}
    In particular,
    \begin{align}
        K(t;0,0)&=Z(t)^2. \\
        \partial_\alpha K(t;0,0)&=\partial_\beta K(t;0,0)=0. \\
        \partial_{\alpha\alpha}K(t;0,0)&=2X(t)^2,\qquad
        \partial_{\beta\beta}K(t;0,0)=2Y(t)^2. \\
        \partial_{\alpha\beta}K(t;0,0)&=2\left(X(t)Y(t)-Z(t)W(t)\right).
    \end{align}

    Since $Z(t)>0$ for $t>0$, the identity
    \begin{equation}
        \partial_{ij}\log K
        =\frac{\partial_{ij}K}{K}
        -\frac{(\partial_iK)(\partial_jK)}{K^2}
    \end{equation}
    gives
    \begin{align}
        \partial_{\alpha\alpha}\log K(t;0,0)
        &=\frac{2X(t)^2}{Z(t)^2},\\
        \partial_{\beta\beta}\log K(t;0,0)
        &=\frac{2Y(t)^2}{Z(t)^2},\\
        \partial_{\alpha\beta}\log K(t;0,0)
        &=\frac{2\left(X(t)Y(t)-Z(t)W(t)\right)}{Z(t)^2}.
    \end{align}
    Differentiation under the Coulson integral is justified
    in a sufficiently small neighborhood of the origin,
    because the differentiated kernels are dominated
    by the corresponding integrable rational functions
    appearing in the definitions of $I_1,I_2,I_3$.
    Therefore, from~\eqref{eq:oo-coulson-expression},
    \begin{align}
        \partial_{\alpha\alpha}E(0,0)
        &=\frac{2}{\pi}
        \int_0^{+\infty}\frac{X(t)^2}{t^2Z(t)^2}\,dt
        =\frac{2}{\pi}I_1,\\
        \partial_{\beta\beta}E(0,0)
        &=\frac{2}{\pi}
        \int_0^{+\infty}\frac{Y(t)^2}{t^2Z(t)^2}\,dt
        =\frac{2}{\pi}I_2,\\
        \partial_{\alpha\beta}E(0,0)
        &=\frac{2}{\pi}
        \int_0^{+\infty}
        \frac{X(t)Y(t)-Z(t)W(t)}{t^2Z(t)^2}\,dt
        =-\frac{2}{\pi}I_3.
    \end{align}
    This proves~\eqref{eq:hessian-matrix}.
    The stated alternatives follow.
\end{proof}

We now present a numerical example.
For each odd triple $(r,s,\ell)$ such that
\begin{equation}
    r,s\in\{3,5,7,9,11,13,15,17\},
    \qquad
    \ell\in\{1,3,5,7,9,11,13,15,17\},
\end{equation}
we performed a numerical search for values
$E(\alpha,\beta)<E(0,0)$
over $(\alpha,\beta)\in[-1,1]^2$
by first performing a grid search with mesh size $0.1$
and then applying $10$ rounds of local refinement,
where in each round the mesh size was halved
and the search was restricted to a neighborhood of
radius two current mesh steps around the best point found so far.
Thus the final mesh width was $0.1/2^{10}$.
Among the $576$ parameter triples in this range,
this search found $477$ cases for which there is a point
$(\alpha,\beta)$ with $E(\alpha,\beta)<E(0,0)$.
It is worth noting that, among these,
$474$ satisfy $\alpha\beta\approx1/4$.
Table~\ref{tab:oo-odd-ell-counterexamples} lists
six representative cases.

\begin{table}[htbp]
    \centering
    \renewcommand{\arraystretch}{1.15}
    \setlength{\tabcolsep}{4pt}
    \begin{tabularx}{\textwidth}{|
        >{\centering\arraybackslash}m{0.06\textwidth}|
        >{\centering\arraybackslash}m{0.06\textwidth}|
        >{\centering\arraybackslash}m{0.06\textwidth}|
        Y|
        >{\centering\arraybackslash}m{0.15\textwidth}|
        >{\centering\arraybackslash}m{0.15\textwidth}|
        >{\centering\arraybackslash}m{0.15\textwidth}|}
    \hline
    $r$ & $s$ & $\ell$ & $(\alpha,\beta)$ & $E(\alpha,\beta)$ & $E(0,0)$ & $\Delta E$ \\
    \hline
    $3$  & $3$  & $1$  & $(0.500000,\,0.500000)$ & $7.841619$  & $7.924777$  & $0.083158$ \\
    \hline
    $3$  & $3$  & $5$  & $(0.197754,\,0.197754)$ & $13.000765$ & $13.000791$ & $0.000026$ \\
    \hline
    $5$  & $3$  & $1$  & $(0.676172,\,0.369727)$ & $10.437006$ & $10.505533$ & $0.068526$ \\
    \hline
    $3$  & $7$  & $3$  & $(0.312500,\,0.800000)$ & $15.562485$ & $15.598262$ & $0.035776$ \\
    \hline
    $3$  & $13$ & $5$  & $(0.250000,\,1.000000)$ & $25.766721$ & $25.789227$ & $0.022507$ \\
    \hline
    $17$ & $17$ & $17$ & $(0.500000,\,0.500000)$ & $64.068103$ & $64.073597$ & $0.005494$ \\
    \hline
    \end{tabularx}
    \caption{Six representative cases with
    $E(\alpha,\beta)<E(0,0)$.
    Here $\Delta E:=E(0,0)-E(\alpha,\beta)$.}
    \label{tab:oo-odd-ell-counterexamples}
\end{table}
\FloatBarrier

Applying Theorem~\ref{thm:oo-hessian}
to the same six triples gives the
following Hessian-test values.

\begin{table}[htbp]
    \centering
    \renewcommand{\arraystretch}{1.15}
    \setlength{\tabcolsep}{4pt}
    \begin{tabularx}{\textwidth}{|
        >{\centering\arraybackslash}m{0.06\textwidth}|
        >{\centering\arraybackslash}m{0.06\textwidth}|
        >{\centering\arraybackslash}m{0.06\textwidth}|
        >{\centering\arraybackslash}m{0.15\textwidth}|
        >{\centering\arraybackslash}m{0.15\textwidth}|
        >{\centering\arraybackslash}m{0.15\textwidth}|
        Y|}
    \hline
    $r$ & $s$ & $\ell$ & $I_1$ & $I_2$ & $I_3$ & $I_3^2-I_1I_2$ \\
    \hline
    $3$  & $3$  & $1$  & $0.522561$ & $0.522561$ & $1.010713$ & $0.748471$ \\
    \hline
    $3$  & $3$  & $5$  & $0.619852$ & $0.619852$ & $0.621954$ & $0.002611$ \\
    \hline
    $5$  & $3$  & $1$  & $0.195887$ & $0.725952$ & $0.775376$ & $0.459004$ \\
    \hline
    $3$  & $7$  & $3$  & $0.815322$ & $0.123086$ & $0.528180$ & $0.178619$ \\
    \hline
    $3$  & $13$ & $5$  & $0.971686$ & $0.044610$ & $0.339653$ & $0.072018$ \\
    \hline
    $17$ & $17$ & $17$ & $0.087627$ & $0.087627$ & $0.121413$ & $0.007063$ \\
    \hline
    \end{tabularx}
    \caption{Hessian-test values for the six examples in
    Table~\ref{tab:oo-odd-ell-counterexamples}.}
    \label{tab:oo-hessian-counterexamples}
\end{table}
\FloatBarrier

In every row of Table~\ref{tab:oo-hessian-counterexamples},
the final column is positive, so Theorem~\ref{thm:oo-hessian}
confirms that $(0,0)$ is not a local minimizer under those cases.

Specifically, we give an algebraic counterexample for $D_{3,3,1}$.

\begin{proposition}\label{prop:d331-certificate}
    For $D_{3,3,1}$, at $\alpha=\beta=1/2$ one has
    \begin{equation}
        E(1/2,1/2)=2+\sqrt{13}+\sqrt5=7.8416192529\ldots,
    \end{equation}
    whereas
    \begin{equation}
        E(0,0)=7.9247772163\ldots.
    \end{equation}
    Therefore $E(1/2,1/2)<E(0,0)$, with gap approximately
    $0.0831579634$.
\end{proposition}
\begin{proof}
    Substituting $r=s=3$ and $\ell=1$ in
    Theorem~\ref{thm:dumbbell-char-poly} gives
    \begin{equation}\label{eq:d331-polynomial-general}
        P_{\alpha,\beta}(\lambda)
        =\lambda^6-7\lambda^4+11\lambda^2-1
        -2(\alpha+\beta)(\lambda^3-3\lambda)+4\alpha\beta.
    \end{equation}
    At $\alpha=\beta=1/2$,
    \begin{equation}
        P_{1/2,1/2}(\lambda)
        =\lambda(\lambda+2)(\lambda^2-\lambda-3)(\lambda^2-\lambda-1).
    \end{equation}
    The sum of the absolute values of the roots of this polynomial
    is $2+\sqrt{13}+\sqrt5$.

    On the other hand,
    \begin{equation}
        P_{0,0}(\lambda)=\lambda^6-7\lambda^4+11\lambda^2-1.
    \end{equation}
    Set $y=\lambda^2$,
    then the three positive values of $y$ are the roots of
    \begin{equation}
        q(y)=y^3-7y^2+11y-1.
    \end{equation}
    A Sturm-sequence verification,
    equivalently rational root isolation, gives
    \begin{equation}
        \begin{aligned}
        y_1&\in(0.09678807,0.09678808),\\
        y_2&\in(2.19393656,2.19393657),\\
        y_3&\in(4.70927535,4.70927537).
        \end{aligned}
    \end{equation}
    Hence
    \begin{equation}
        E(0,0)=2(\sqrt{y_1}+\sqrt{y_2}+\sqrt{y_3})
        \in(7.9247771945,7.9247772427),
    \end{equation}
    while $2+\sqrt{13}+\sqrt5<7.8416192530$. This proves the strict
    inequality.
\end{proof}

Figure~\ref{fig:d331-diagonal-gap} plots the diagonal energy gap
$E(0,0)-E(a,a)$ for $D_{3,3,1}$.
It also indicates that nonsmooth points must be handled separately,
as the polynomial $P_{1/2,1/2}$ has a zero root at the point $a=1/2$,
and hence the energy is not differentiable there.

\begin{figure}[htbp]
    \centering
    \begin{tikzpicture}
    \begin{axis}[
        width=0.76\textwidth,
        height=0.42\textwidth,
        xmin=0,xmax=1,
        ymin=-0.40,ymax=0.11,
        xlabel={$a$ in the diagonal slice $\alpha=\beta=a$},
        ylabel={$E(0,0)-E(a,a)$},
        tick label style={font=\small},
        label style={font=\small},
        grid=both,
        minor tick num=1,
        major grid style={line width=0.2pt,draw=gray!35},
        minor grid style={line width=0.1pt,draw=gray!20},
        axis lines=left,
        scaled ticks=false,
        legend style={draw=none,font=\small,at={(0.03,0.97)},anchor=north west}
    ]
        \addplot[black,thick,mark=none] coordinates {
            (0.000,0.000000) (0.025,0.000194) (0.050,0.000777)
            (0.075,0.001751) (0.100,0.003116) (0.125,0.004875)
            (0.150,0.007032) (0.175,0.009592) (0.200,0.012558)
            (0.225,0.015937) (0.250,0.019737) (0.275,0.023964)
            (0.300,0.028627) (0.325,0.033738) (0.350,0.039306)
            (0.375,0.045345) (0.400,0.051869) (0.425,0.058894)
            (0.450,0.066437) (0.475,0.074517) (0.500,0.083158)
            (0.525,0.058859) (0.550,0.034769) (0.575,0.010884)
            (0.600,-0.012802) (0.625,-0.036293) (0.650,-0.059592)
            (0.675,-0.082704) (0.700,-0.105633) (0.725,-0.128382)
            (0.750,-0.150956) (0.775,-0.173358) (0.800,-0.195590)
            (0.825,-0.217658) (0.850,-0.239563) (0.875,-0.261308)
            (0.900,-0.282898) (0.925,-0.304334) (0.950,-0.325620)
            (0.975,-0.346758) (1.000,-0.367752)
        };
        \addlegendentry{diagonal gap}
        \addplot[black,dashed,domain=0:1,samples=2] {0};
        \addplot[only marks,mark=*,mark size=2.2pt] coordinates {(0.5,0.083158)};
        \node[anchor=south east,font=\scriptsize] at (axis cs:0.5,-0.01) {$a=1/2$};
        \node[anchor=north west,font=\scriptsize] at (axis cs:0.61,0) {zero of the gap};
    \end{axis}
    \end{tikzpicture}
    \caption{Diagonal energy gap for $D_{3,3,1}$.
    Positive values mean that the diagonal point
    $(a,a)$ improves upon $(0,0)$.}
    \label{fig:d331-diagonal-gap}
\end{figure}
\FloatBarrier

\section{The Mixed-Parity Case}\label{sec:mixed}
\subsection{Notation}
We now assume that one cycle is even and the other is odd.
Since $D_{r,s,\ell}\cong D_{s,r,\ell}$
after appropriately relabeling the two cycles,
it is enough to treat the case in which $r$ is even and $s$ is odd.

For later convenience, write
\begin{equation}
    \alpha:=(-1)^{r/2}\rp(\gamma_r),
    \qquad
    \beta:=(-1)^{(s+1)/2}\rp(\gamma_s).
\end{equation}
Then $\alpha,\beta\in[-1,1]$.

\subsection{Simplified Coulson Formula}
\begin{lemma}\label{lem:mixed-kernel}
    Let $\Phi$ be a complex unit gain graph on $D_{r,s,\ell}$
    where $r$ is even and $s$ is odd.
    For each $k\ge 0$, define matching numbers
    \begin{equation}\label{eq:mixed-coefficient-definition}
        \begin{aligned}
            &M_k=m(D_{r,s,\ell},k),\qquad
            &T_k=m(P_{\ell-1},k-(r+s-1)/2), \\
            &R_k=m(D_{r,s,\ell}-C_r,k-r/2),\qquad
            &S_k=m(D_{r,s,\ell}-C_s,k-(s-1)/2),
        \end{aligned}
    \end{equation}
    where out-of-range matching indices are interpreted as $0$.
    For $t>0$, put polynomials
    \begin{equation}\label{eq:mixed-four-functions}
        \begin{aligned}
        Z(t)&:=\sum_{k=0}^{\lfloor n/2\rfloor} M_k t^{2k},
        &
        W(t)&:=\sum_{k=(r+s-1)/2}^{\lfloor (n-1)/2\rfloor} T_k t^{2k+1},\\
        X(t)&:=\sum_{k=r/2}^{\lfloor n/2\rfloor} R_k t^{2k},
        &
        Y(t)&:=\sum_{k=(s-1)/2}^{\lfloor (n-1)/2\rfloor} S_k t^{2k+1}.
        \end{aligned}
    \end{equation}
    Then $Z(t),X(t),Y(t),W(t)$ are positive on $(0,+\infty)$.
    If we further set
    \begin{equation}\label{eq:mixed-kernel}
        K(t;\alpha,\beta)
        =\left(Z(t)-2\alpha X(t)\right)^2+4\beta^2\left(Y(t)-2\alpha W(t)\right)^2,
    \end{equation}
    then we have
    \begin{equation}\label{eq:mixed-coulson-expression}
        E(\Phi)=\frac{1}{\pi}\int_0^{+\infty}
        \frac1{t^2}\log K(t;\alpha,\beta)\,dt.
    \end{equation}
\end{lemma}
\begin{proof}
    Recall the Coulson formula in the form
    \begin{equation}\label{eq:mixed-coulson-raw}
        E(\Phi)=\frac1\pi\int_0^{+\infty} \frac1{t^2}
        \log\left|t^nP_\Phi(i/t)\right|^2\,dt.
    \end{equation}
    It suffices to compute $t^nP_\Phi(i/t)$ explicitly.

    Recall the characteristic polynomial expression
    from Theorem~\ref{thm:dumbbell-char-poly},
    \begin{equation}\label{eq:charpoly-dumbbell-2}
        P_\Phi(x)=m_{D_{r,s,\ell}}(x)
        -2a\,m_{D_{r,s,\ell}-C_r}(x)
        -2b\,m_{D_{r,s,\ell}-C_s}(x)
        +4ab\,m_{P_{\ell-1}}(x),
    \end{equation}
    where $a=\rp(\gamma_r)$ and $b=\rp(\gamma_s)$.

    We compute the four terms separately.
    First, by the definition of the matching polynomial,
    \begin{equation}
        m_{D_{r,s,\ell}}(x)
        =\sum_{j=0}^{\lfloor n/2\rfloor}(-1)^j m(D_{r,s,\ell},j)x^{n-2j}.
    \end{equation}
    Hence
    \begin{equation}
        \begin{aligned}
            t^n m_{D_{r,s,\ell}}(i/t)
            & =\sum_{j=0}^{\lfloor n/2\rfloor}(-1)^j m(D_{r,s,\ell},j)i^{n-2j}t^{2j} \\
            & =i^n\sum_{j=0}^{\lfloor n/2\rfloor} m(D_{r,s,\ell},j)t^{2j}
            =i^n\sum_{k=0}^{\lfloor n/2\rfloor}M_k t^{2k}
            =i^nZ(t).
        \end{aligned}
    \end{equation}

    Next, since $r$ is even, writing $k=j+r/2$ gives
    \begin{equation}
        m_{D_{r,s,\ell}-C_r}(x)
        =\sum_{k=r/2}^{\lfloor n/2\rfloor}(-1)^{k-r/2}R_k x^{n-2k}.
    \end{equation}
    Therefore
    \begin{equation}
        \begin{aligned}
            t^n m_{D_{r,s,\ell}-C_r}(i/t)
            & =\sum_{k=r/2}^{\lfloor n/2\rfloor}(-1)^{k-r/2}R_k i^{n-2k} t^{2k} \\
            & =i^n(-1)^{r/2}\sum_{k=r/2}^{\lfloor n/2\rfloor}R_k t^{2k}
            =i^n(-1)^{r/2}X(t).
        \end{aligned}
    \end{equation}
    Since $\alpha=(-1)^{r/2}a$, it follows that
    \begin{equation}\label{eq:mixed-r-term}
        -2a\,t^n m_{D_{r,s,\ell}-C_r}(i/t)=i^n(-2\alpha X(t)).
    \end{equation}

    Similarly, since $s$ is odd, writing $k=j+(s-1)/2$ yields
    \begin{equation}
        m_{D_{r,s,\ell}-C_s}(x)
        =\sum_{k=(s-1)/2}^{\lfloor (n-1)/2\rfloor}(-1)^{k-(s-1)/2}S_k x^{n-(2k+1)}.
    \end{equation}
    Hence
    \begin{equation}
        \begin{aligned}
            t^n m_{D_{r,s,\ell}-C_s}(i/t)
            & =\sum_{k=(s-1)/2}^{\lfloor (n-1)/2\rfloor}(-1)^{k-(s-1)/2}S_k i^{n-(2k+1)} t^{2k+1} \\
            & =i^n\left((-1)^{(s+1)/2}i\right)\sum_{k=(s-1)/2}^{\lfloor (n-1)/2\rfloor}S_k t^{2k+1}
            =i^n\left((-1)^{(s+1)/2}i\right)Y(t).
        \end{aligned}
    \end{equation}
    By $\beta=(-1)^{(s+1)/2}b$, we obtain
    \begin{equation}\label{eq:mixed-s-term}
        -2b\,t^n m_{D_{r,s,\ell}-C_s}(i/t)=i^n(-2i\beta Y(t)).
    \end{equation}

    Finally, because $r+s$ is odd and $|V(P_{\ell-1})|=\ell-1=n-r-s$,
    setting $k=j+(r+s-1)/2$ gives
    \begin{equation}
        m_{P_{\ell-1}}(x)
        =\sum_{k=(r+s-1)/2}^{\lfloor (n-1)/2\rfloor}(-1)^{k-(r+s-1)/2}T_k x^{n-(2k+1)}.
    \end{equation}
    Therefore
    \begin{equation}
        \begin{aligned}
            t^n m_{P_{\ell-1}}(i/t)
            & =\sum_{k=(r+s-1)/2}^{\lfloor (n-1)/2\rfloor}(-1)^{k-(r+s-1)/2}T_k i^{n-(2k+1)} t^{2k+1} \\
            & =i^n\left((-1)^{(r+s+1)/2}i\right)\sum_{k=(r+s-1)/2}^{\lfloor (n-1)/2\rfloor}T_k t^{2k+1} \\
            & =i^n\left((-1)^{r/2}(-1)^{(s+1)/2}i\right)W(t).
        \end{aligned}
    \end{equation}
    Using $\alpha=(-1)^{r/2}a$ and $\beta=(-1)^{(s+1)/2}b$, we get
    \begin{equation}\label{eq:mixed-t-term}
        4ab\,t^n m_{P_{\ell-1}}(i/t)=i^n(4i\alpha\beta W(t)).
    \end{equation}

    Substituting~\eqref{eq:mixed-r-term},~\eqref{eq:mixed-s-term},
    and~\eqref{eq:mixed-t-term} into~\eqref{eq:charpoly-dumbbell-2},
    we obtain
    \begin{equation}
        \begin{aligned}
            t^nP_\Phi(i/t)
            & =i^n\left(Z(t)-2\alpha X(t)-2i\beta Y(t)+4i\alpha\beta W(t)\right) \\
            & =i^n\left(Z(t)-2\alpha X(t)-2i\beta\left(Y(t)-2\alpha W(t)\right)\right),
        \end{aligned}
    \end{equation}
    Taking absolute values, we have
    \begin{equation}
        \left|t^nP_\Phi(i/t)\right|^2
        =\left(Z(t)-2\alpha X(t)\right)^2
        +4\beta^2\left(Y(t)-2\alpha W(t)\right)^2,
    \end{equation}
    which is exactly~\eqref{eq:mixed-kernel}.
    Inserting this into~\eqref{eq:mixed-coulson-raw}
    immediately proves~\eqref{eq:mixed-coulson-expression}.

    It remains to prove that
    $Z(t),X(t),Y(t),W(t)$ are all positive on $(0,+\infty)$.
    All coefficients in~\eqref{eq:mixed-four-functions}
    are matching numbers, hence are nonnegative.
    Moreover, notice that
    \begin{equation}
        M_0=1,\qquad R_{r/2}=1,\qquad S_{(s-1)/2}=1,\qquad T_{(r+s-1)/2}=1.
    \end{equation}
    Therefore, for every $t>0$,
    \begin{equation}
        Z(t)\ge 1>0,\qquad
        X(t)\ge t^r>0,\qquad
        Y(t)\ge t^s>0,\qquad
        W(t)\ge t^{r+s}>0.
    \end{equation}
    This completes the proof.
\end{proof}

\subsection{Extremal Gain Assignments for the Maximum and Minimum Energy}
The maximum and minimum energy problems in the mixed-parity case
are completely determined.

\begin{theorem}
    Let $r$ be even and $s$ be odd.
    Then the maximum of $E(\Phi)$ is attained exactly at
    \begin{equation}
        (\alpha,\beta)=(-1,1)\quad\text{or}\quad
        (\alpha,\beta)=(-1,-1).
    \end{equation}
\end{theorem}
\begin{proof}
    Fix $\alpha\in[-1,1]$. From~\eqref{eq:mixed-kernel},
    \begin{equation}
        K(t;\alpha,\beta)=\left(Z(t)-2\alpha X(t)\right)^2+4\beta^2\left(Y(t)-2\alpha W(t)\right)^2
        \le K(t;\alpha,1)
    \end{equation}
    for every $t>0$, since $\beta^2\le 1$.
    From the definitions of $Y(t)$ and $W(t)$
    and the facts that $S_{(s-1)/2}=1$ and $T_{(r+s-1)/2}=1$, we have
    \begin{equation}
        Y(t)=t^s+O(t^{s+2}),\qquad W(t)=t^{r+s}+O(t^{r+s+2})
        \quad (t\to0^+).
    \end{equation}
    Since $r\ge 2$, it follows that $W(t)=o(t^s)$. Hence
    \begin{equation}
        Y(t)-2\alpha W(t)=t^s(1+o(1))\ne 0
    \end{equation}
    for all sufficiently small $t>0$,
    uniformly in $\alpha\in[-1,1]$. Therefore, if $|\beta|<1$, then
    \begin{equation}
        K(t;\alpha,\beta)<K(t;\alpha,1)
    \end{equation}
    for all sufficiently small $t>0$.
    Hence every maximum-energy point must satisfy
    $\beta^2=1\Leftrightarrow|\beta|=1$.

    It remains to maximize
    \begin{equation}
        K(t;\alpha,\pm 1)
        =\left(Z(t)-2\alpha X(t)\right)^2+4\left(Y(t)-2\alpha W(t)\right)^2.
    \end{equation}
    Expanding the right-hand side gives
    \begin{equation}
        K(t;\alpha,\pm 1)
        =4\left(X(t)^2+4W(t)^2\right)\alpha^2
        -4\left(Z(t) X(t)+4Y(t) W(t)\right)\alpha
        +Z(t)^2+4Y(t)^2.
    \end{equation}
    Since $X(t)^2+4W(t)^2>0$,
    this quadratic function of $\alpha$ is strictly convex.
    Therefore its maximum on $[-1,1]$ is attained at an endpoint.
    Now, since $Z(t),X(t),Y(t),W(t)$ are positive on $(0,+\infty)$,
    \begin{equation}
        K(t;-1,\pm1)-K(t;1,\pm1)
        =8\left(Z(t)X(t)+4Y(t)W(t)\right)>0.
    \end{equation}
    Hence for every $t>0$,
    \begin{equation}
        K(t;\alpha,\beta)\le K(t;-1,\pm1),
    \end{equation}
    with equality only when $\alpha=-1$ and $|\beta|=1$.
    This completes the proof.
\end{proof}

\begin{theorem}
    Let $r$ be even and $s$ be odd.
    Then the minimum of $E(\Phi)$ is attained exactly at
    \begin{equation}
        (\alpha,\beta)=(1,0).
    \end{equation}
\end{theorem}
\begin{proof}
    Fix $\alpha\in[-1,1]$. From~\eqref{eq:mixed-kernel},
    \begin{equation}
        K(t;\alpha,\beta)=\left(Z(t)-2\alpha X(t)\right)^2+4\beta^2\left(Y(t)-2\alpha W(t)\right)^2
        \ge K(t;\alpha,0)
    \end{equation}
    for every $t>0$,
    and as in the proof of the maximum energy statement,
    the inequality is strict for every $\beta\ne 0$
    and all sufficiently small $t$.
    Hence every minimizer must satisfy $\beta=0$.

    It remains to minimize
    \begin{equation}
        K(t;\alpha,0)=\left(Z(t)-2\alpha X(t)\right)^2.
    \end{equation}
    Each matching of $D_{r,s,\ell}-C_r$
    has at least two extensions to a matching of $D_{r,s,\ell}$,
    obtained by choosing one of the two perfect matchings of $C_r$.
    Therefore
    \begin{equation}
        m(D_{r,s,\ell},k)\ge 2m(D_{r,s,\ell}-C_r,k-r/2)
    \end{equation}
    for all $k$. This implies $Z(t)\ge 2X(t)$ for every $t>0$.
    Hence, for every $\alpha\le 1$,
    \begin{equation}
        Z(t)-2\alpha X(t)\ge Z(t)-2X(t)\ge 0,
    \end{equation}
    so
    \begin{equation}
        K(t;\alpha,0)\ge K(t;1,0).
    \end{equation}
    The inequality is strict when $\alpha<1$, because $X(t)>0$.
    This completes the proof.
\end{proof}

\section{Conclusion}\label{sec:conclusion}
The extremal energy problem for complex unit gain dumbbell graphs
can be summarized in Table~\ref{tab:summary}.

\begin{table}[htbp]
\centering
\small
\setlength{\tabcolsep}{6pt}
\renewcommand{\arraystretch}{1.5}

\begin{tabularx}{\linewidth}{|C{2.6cm}|Y|Y|Y|}
    \hline
    \textbf{Case} & \textbf{Parameters} & \textbf{Maximum energy} & \textbf{Minimum energy} \\
    \hline
    $r,s$ even
    &
    $\begin{aligned}
    \alpha&=(-1)^{r/2}\rp(\gamma_r),\\
    \beta&=(-1)^{s/2}\rp(\gamma_s)
    \end{aligned}$
    &
    attained exactly at \newline
    $(\alpha,\beta)=(-1,-1)$
    &
    attained exactly at \newline
    $(\alpha,\beta)=(1,1)$
    \\
    \hline
    $r,s$ odd, $\ell$ even
    &
    \multirow{2}{=}{\centering
    $\begin{aligned}
    \alpha&=(-1)^{(r+1)/2}\rp(\gamma_r),\\
    \beta&=(-1)^{(s+1)/2}\rp(\gamma_s)
    \end{aligned}$}
    &
    attained exactly at \newline
    $(\alpha,\beta)=(1,1)$ or $(\alpha,\beta)=(-1,-1)$
    &
    attained exactly at \newline
    $(\alpha,\beta)=(0,0)$
    \\
    \cline{1-1}\cline{3-4}
    $r,s$ odd, $\ell$ odd
    &
    &
    attained exactly at \newline
    $(\alpha,\beta)=(1,-1)$ or $(\alpha,\beta)=(-1,1)$
    &
    not yet fully classified
    \\
    \hline
    $r$ even, $s$ odd
    &
    $\begin{aligned}
    \alpha&=(-1)^{r/2}\rp(\gamma_r),\\
    \beta&=(-1)^{(s+1)/2}\rp(\gamma_s)
    \end{aligned}$
    &
    attained exactly at \newline
    $(\alpha,\beta)=(-1,1)$ or $(\alpha,\beta)=(-1,-1)$
    &
    attained exactly at \newline
    $(\alpha,\beta)=(1,0)$
    \\
    \hline
    \end{tabularx}

    \caption{Summary of the extremal energy problem
    for complex unit gain dumbbell graphs $D_{r,s,\ell}$.
    Here $\gamma_r=\vphi(C_r)$ and $\gamma_s=\vphi(C_s)$
    denote the gains of the two cycles.}
    \label{tab:summary}
\end{table}
\FloatBarrier

The argument of the comparison of the coefficients
in the even-even case
also suggests a natural extension
to complex unit gain cactus graphs
with several pairwise vertex-disjoint even cycles.
We leave this as a possible direction of future work.

\bibliographystyle{elsarticle-num}
\bibliography{ref}

@misc{Gutman1978,
    author = {Gutman, Ivan},
    title = {The energy of a graph},
    year = {1978},
    language = {English},
    howpublished = {Ber. {Math}.-{Stat}. {Sekt}. {Forschungszent}. {Graz} 103, 22 {S}.},
    zbMATH = {3623599},
    Zbl = {0402.05040}
}

@article{Samanta2025,
	title = {Bounds and extremal graphs for the energy of complex unit gain graphs},
	volume = {721},
	issn = {00243795},
	_url = {https://linkinghub.elsevier.com/retrieve/pii/S0024379524001241},
	doi = {10.1016/j.laa.2024.03.028},
	language = {en},
	urldate = {2026-03-21},
	journal = {Linear Algebra and its Applications},
	author = {Samanta, Aniruddha and Rajesh Kannan, M.},
	month = sep,
	year = {2025},
	pages = {844--866},
}

@article{Samanta2021,
	title = {Bounds for the energy of a complex unit gain graph},
	volume = {612},
	issn = {00243795},
	_url = {https://linkinghub.elsevier.com/retrieve/pii/S0024379520305449},
	doi = {10.1016/j.laa.2020.11.019},
	language = {en},
	urldate = {2026-03-21},
	journal = {Linear Algebra and its Applications},
	author = {Samanta, Aniruddha and Kannan, M. Rajesh},
	month = mar,
	year = {2021},
	pages = {1--29},
}

@article{Zaslavsky1989,
	title = {Biased graphs. {I}. {Bias}, balance, and gains},
	volume = {47},
	copyright = {https://www.elsevier.com/tdm/userlicense/1.0/},
	issn = {00958956},
	_url = {https://linkinghub.elsevier.com/retrieve/pii/0095895689900634},
	doi = {10.1016/0095-8956(89)90063-4},
	language = {en},
	number = {1},
	urldate = {2025-12-29},
	journal = {Journal of Combinatorial Theory, Series B},
	author = {Zaslavsky, Thomas},
	month = aug,
	year = {1989},
	pages = {32--52},
}

@article{Harary1953,
    author = {Frank Harary},
    title = {{On the notion of balance of a signed graph.}},
    volume = {2},
    journal = {Michigan Mathematical Journal},
    number = {2},
    publisher = {University of Michigan, Department of Mathematics},
    pages = {143--146},
    year = {1953},
    doi = {10.1307/mmj/1028989917},
    _URL = {https://doi.org/10.1307/mmj/1028989917}
}

@article{Bhat2017,
	title = {Unicyclic signed graphs with minimal energy},
	volume = {226},
	issn = {0166218X},
	_url = {https://linkinghub.elsevier.com/retrieve/pii/S0166218X17301518},
	doi = {10.1016/j.dam.2017.03.015},
	language = {en},
	urldate = {2026-03-21},
	journal = {Discrete Applied Mathematics},
	author = {Bhat, Mushtaq A. and Pirzada, S.},
	month = jul,
	year = {2017},
	pages = {32--39},
}

@article{Bhat2018,
	title = {Bicyclic signed graphs with minimal and second minimal energy},
	volume = {551},
	issn = {00243795},
	_url = {https://linkinghub.elsevier.com/retrieve/pii/S0024379518301782},
	doi = {10.1016/j.laa.2018.03.047},
	language = {en},
	urldate = {2026-03-21},
	journal = {Linear Algebra and its Applications},
	author = {Bhat, Mushtaq A. and Samee, U. and Pirzada, S.},
	month = aug,
	year = {2018},
	pages = {18--35},
}

@article{Wang2020,
	title = {Research on Extreme Signed Graphs with Minimal Energy in Tricyclic Signed Graphs {$S(n,n+2)$}},
	volume = {2020},
	copyright = {http://creativecommons.org/licenses/by/4.0/},
	issn = {1076-2787, 1099-0526},
	_url = {https://www.hindawi.com/journals/complexity/2020/8747684/},
	doi = {10.1155/2020/8747684},
	language = {en},
	urldate = {2026-03-21},
	journal = {Complexity},
	author = {Wang, Yajing and Gao, Yubin},
	month = aug,
	year = {2020},
	pages = {1--8},
}

@article{Reff2012,
	title = {Spectral properties of complex unit gain graphs},
	volume = {436},
	copyright = {https://www.elsevier.com/tdm/userlicense/1.0/},
	issn = {00243795},
	_url = {https://linkinghub.elsevier.com/retrieve/pii/S002437951100718X},
	doi = {10.1016/j.laa.2011.10.021},
	language = {en},
	number = {9},
	urldate = {2025-12-18},
	journal = {Linear Algebra and its Applications},
	author = {Reff, Nathan},
	month = may,
	year = {2012},
	pages = {3165--3176},
}

@article{Mateljevic2010,
	title = {Energy of a polynomial and the {Coulson} integral formula},
	volume = {48},
	copyright = {http://www.springer.com/tdm},
	issn = {0259-9791, 1572-8897},
	_url = {http://link.springer.com/10.1007/s10910-010-9725-z},
	doi = {10.1007/s10910-010-9725-z},
	language = {en},
	number = {4},
	urldate = {2026-03-22},
	journal = {Journal of Mathematical Chemistry},
	author = {Mateljević, Miodrag and Božin, Vladimir and Gutman, Ivan},
	month = nov,
	year = {2010},
	pages = {1062--1068},
}

@article{Wang2010,
	title = {Spectral Characterizations of Dumbbell Graphs},
	volume = {17},
	issn = {1077-8926},
	_url = {https://www.combinatorics.org/ojs/index.php/eljc/article/view/v17i1r42},
	doi = {10.37236/314},
	language = {en},
	number = {1},
	urldate = {2026-03-22},
	journal = {The Electronic Journal of Combinatorics},
	author = {Wang, Jianfeng and Belardo, Francesco and Huang, Qiongxiang and Marzi, Enzo M. Li},
	month = mar,
	year = {2010},
	pages = {R42},
}

@article{Reff2016,
	title = {Oriented gain graphs, line graphs and eigenvalues},
	volume = {506},
	issn = {00243795},
	_url = {https://linkinghub.elsevier.com/retrieve/pii/S0024379516302154},
	doi = {10.1016/j.laa.2016.05.040},
	language = {en},
	urldate = {2025-12-27},
	journal = {Linear Algebra and its Applications},
	author = {Reff, Nathan},
	month = oct,
	year = {2016},
	pages = {316--328},
}

@article{Mehatari2022,
	title = {On the adjacency matrix of a complex unit gain graph},
	volume = {70},
	issn = {0308-1087, 1563-5139},
	_url = {https://www.tandfonline.com/doi/full/10.1080/03081087.2020.1776672},
	doi = {10.1080/03081087.2020.1776672},
	language = {en},
	number = {9},
	urldate = {2025-12-18},
	journal = {Linear and Multilinear Algebra},
	author = {Mehatari, Ranjit and Kannan, M. Rajesh and Samanta, Aniruddha},
	month = jun,
	year = {2022},
	pages = {1798--1813},
}

@article{Gutman1979,
	author = {Ivan Gutman},
	title = {The Matching Polynomial},
	journal = {MATCH Communications in Mathematical and in Computer Chemistry},
	volume = {6},
	pages = {75--91},
	year = {1979},
	_url = {https://match.pmf.kg.ac.rs/electronic_versions/Match06/match6_75-91.pdf}
}

@article{Cigler2003,
	title = {{$q$}-{Fibonacci} Polynomials},
	volume = {41},
	issn = {0015-0517, 2641-340X},
	_url = {https://www.tandfonline.com/doi/full/10.1080/00150517.2003.12428602},
	doi = {10.1080/00150517.2003.12428602},
	language = {en},
	number = {1},
	urldate = {2026-03-22},
	journal = {The Fibonacci Quarterly},
	author = {Cigler, Johann},
	month = feb,
	year = {2003},
	pages = {31--40},
}

@article{Gutman1986,
	author = {Ivan Gutman},
	title = {Some Relations for the {$\mu$}-Polynomial},
	journal = {MATCH Communications in Mathematical Chemistry},
	volume = {19},
	pages = {127--137},
	year = {1986},
	_url = {https://match.pmf.kg.ac.rs/electronic_versions/Match19/match19_127-137.pdf}
}

@article{Wissing2025,
	title = {Symmetry in complex unit gain graphs and their spectra},
	volume = {722},
	issn = {00243795},
	_url = {https://linkinghub.elsevier.com/retrieve/pii/S0024379525002289},
	doi = {10.1016/j.laa.2025.05.012},
	language = {en},
	urldate = {2025-12-18},
	journal = {Linear Algebra and its Applications},
	author = {Wissing, Pepijn and Van Dam, Edwin R.},
	month = oct,
	year = {2025},
	pages = {164--177},
}

@article{Samanta2024,
  author = {Samanta, Aniruddha and Kannan, M. Rajesh},
  title = {On the spectrum of complex unit gain graph},
  journal = {Journal of the Ramanujan Mathematical Society},
  volume = {39},
  number = {2},
  year = {2024},
  pages = {131--142},
  doi = {10.48550/arXiv.1908.10668}
}

@article{Belardo2023,
	title = {{NEPS} of complex unit gain graphs},
	volume = {39},
	issn = {1081-3810},
	_url = {https://journals.uwyo.edu/index.php/ela/article/view/8015},
	doi = {10.13001/ela.2023.8015},
	language = {en},
	urldate = {2026-03-22},
	journal = {The Electronic Journal of Linear Algebra},
	author = {Belardo, Francesco and Brunetti, Maurizio and Khan, Suliman},
	month = dec,
	year = {2023},
	pages = {621--643}
}

@article{Furtula2008,
	title = {Bicyclic molecular graphs with the greatest energy},
	volume = {73},
	copyright = {http://creativecommons.org/licenses/by-nc-nd/4.0/},
	issn = {0352-5139, 1820-7421},
	_url = {https://doiserbia.nb.rs/Article.aspx?ID=0352-51390804431F},
	doi = {10.2298/JSC0804431F},
	language = {en},
	number = {4},
	urldate = {2026-03-23},
	journal = {Journal of the Serbian Chemical Society},
	author = {Furtula, Boris and Radenkovic, Slavko and Gutman, Ivan},
	year = {2008},
	pages = {431--433},
}

\end{document}